\documentclass[12pt, reqno]{amsart}
\usepackage{amssymb}
\usepackage{amsthm}
\usepackage{mathtools}
\usepackage[normalem]{ulem}
\usepackage{amsmath, geometry}
\usepackage{listings}
\usepackage[utf8]{inputenc}
\usepackage[english]{babel}
\usepackage[dvipsnames]{xcolor}
\usepackage{amsfonts}
\usepackage{lipsum}

\usepackage{tikz}
\usetikzlibrary{cd}
\usepackage{mathrsfs}
\usetikzlibrary{arrows}
\usepackage{scrextend}
\usepackage[makeroom]{cancel}
\usepackage{enumitem}
\usepackage{float}% to use H for figures
\usepackage{graphicx} % Allows including images

\usepackage[backend=biber,
    style=numeric,
    maxbibnames=99,
    giveninits=true,
    url=false,        
    doi=false,        
    isbn=false,      
    eprint=true
    ]{biblatex}
\definecolor{qqffff}{rgb}{0,1,1}
\definecolor{ttffcc}{rgb}{0.2,1,0.8}
\definecolor{ududff}{rgb}{0.30196078431372547,0.30196078431372547,1}
\definecolor{qqqqff}{rgb}{0,0,1}
\definecolor{ffffff}{rgb}{1,1,1}
\definecolor{ffqqqq}{rgb}{1,0,0}
\definecolor{qqffqq}{rgb}{0,1,0}
\definecolor{cqcqcq}{rgb}{0.7529411764705882,0.7529411764705882,0.7529411764705882}
\definecolor{xfqqff}{rgb}{0.4980392156862745,0,1}
\definecolor{ccwwff}{rgb}{0.8,0.4,1}
\definecolor{zzttqq}{rgb}{0.6,0.2,0}
\definecolor{cczzff}{rgb}{0.8,0.4,1}
\theoremstyle{definition}
\newtheorem*{assumptions}{Assumptions}

\theoremstyle{plain}
\newtheorem{theorem}{Theorem}[section]
\newtheorem{lemma}[theorem]{Lemma}
\newtheorem{corollary}[theorem]{Corollary}
\newtheorem{proposition}[theorem]{Proposition}

\theoremstyle{definition}
\newtheorem{example}[theorem]{Example}
\newtheorem{fact}[theorem]{Fact}
\newtheorem{notation}[theorem]{Notation}
\newtheorem{definition}[theorem]{Definition}

\newcommand{\br}{\mbox{branch}}
\newcommand{\hei}{\mbox{height}}
\newcommand{\dist}{\text{dist}}
\newcommand{\pth}{\mbox{path}}

\newcommand\tcr[1]{\textcolor{red}{#1}}
\newcommand\tcb[1]{\textcolor{blue}{#1}}

\newcommand{\hh}{\mbox{height}}

\newcommand{\R}{\mathcal{R}}
\newcommand{\B}{\mathcal{B}}

\newcommand\Set[2]{\left\lbrace #1 \mid #2 \right\rbrace}
\newcommand\set[1]{\left\lbrace #1 \right\rbrace}

\newcommand{\calD}{\mathcal{D}}

\allowdisplaybreaks

\begin{document}

\dedicatory{To the memory of Dr. Renu C. Laskar}

\title[Characterization of Well-Totally Dominated Trees]{Characterization of Well-Totally Dominated Trees}

\author[Lim, et al.]{Jounglag Lim}
\address{School of mathematical and statistical sciences, Clemson University, Clemson, SC 29634}
\email{joungll@clemson.edu}

\author[]{James Gossell}
\address{Department of mathematics and statistics, University of Alaska, Fairbanks, AK 99775}
\email{jegossell@alaska.edu}

\author[]{Keri Ann Sather-Wagstaff}
\address{Framingham State University, Framingham, MA 01701}
\email{ksatherwagstaff@framingham.edu}

\author[]{Devin Adams}
%\address{}
\email{devin.adams@maine.edu}

\author[]{Suzanna Castro-Tarabulsi}
\address{Mathematical Sciences, George Mason University,  Fairfax, VA 22030}
\email{starabul@gmu.edu}

\author[]{Aayahna Herbert}
\address{School of mathematical and statistical sciences, Clemson University, Clemson, SC 29634}
\email{aayahna.herbert@gmail.com}

\author[]{Vi Anh Nguyen}
\address{Department of Mathematics, University of Illinois Urbana-Champaign, Urbana, IL 61801}
\email{vianhan2@illinois.edu}

\author[]{Yifan Qian}
%\address{}
\email{qian.410@buckeyemail.osu.edu}

\author[]{Matthew Schaller}
\address{School of mathematical and statistical sciences, Clemson University, Clemson, SC 29634}
\email{mwschaller@gmail.com}

\author[]{Zoe Zhou}
\address{Columbia university irving medical center, Columbia University, New York, NY 10032}
\email{zz2992@cumc.columbia.edu}

\author[]{Yuyang Zhuo}
\address{Department of mathematics, The Ohio State University, Columbus, OH 43210}
\email{huo.48@buckeyemail.osu.edu}

\begin{abstract}
    Let $G$ be a graph with no isolated vertices.
    A set of vertices $S$ is a total dominating set (TDS) if every vertex in $G$ is adjacent to at least one vertex in $S$.
    We say $G$ is well-totally dominated (WTD) if every minimal TDS has the same size.
    In this paper, we present two characterizations of well-totally dominated trees, one being descriptive and the other being constructive.
    In particular, our characterizations imply that it takes only polynomial time to verify whether a given tree is WTD.
\end{abstract}

\maketitle
\setcounter{tocdepth}{1}

\section{Introduction}

Let $G = (V,E)$ be a graph with no isolated vertices.
For $v \in V$, the \emph{open neighborhood} of $v$ is the set of vertices that are adjacent to $v$ in $G$.
In other words,
$N_G(v) = N(v) = \Set{u \in V}{ \set{u,v} \in E}$.
For $S \subseteq V$, we define the open neighborhood of $S$ to be the set
$N_G(S) = N(S) = \bigcup_{v \in S}N(v)$.
A subset $S \subseteq V$ is called a \emph{total dominating set} (TDS) in $G$ if $N(S) = V$.
We say that $S$ is a minimal TDS if no proper subset of $S$ is also a TDS.
Total domination in graphs has been studied extensively in the literature \cite{TDinG}.
In particular, computing the size of the smallest TDS (called the \emph{total domination number}, denoted $\gamma_t(G)$) is an active research area; similarly, computing the size of the largest minimal TDS in $G$ (called the \emph{upper total domination number}, denoted $\Gamma_t(G)$) is of interest as well.
It is well-known that computing $\gamma_t(G)$ and $\Gamma_t(G)$ are NP-hard problems \cite{pfaff1984np,MR2043923}.
A way to get around this is to consider graphs whose minimal TDS are also minimum TDS (with respect to size); i.e., $\gamma_t(G) = \Gamma_t(G)$.
Such graphs are called \emph{well-totally dominated} (WTD).

A set $S \subseteq V$ is a \emph{dominating set} if $N(S)\supseteq V\setminus S$ (and a minimal dominating set is defined similarly to a minimal TDS), and $G$ is \emph{well-dominated} if every minimal dominating set has the same size.
While well-dominated graphs have been actively studied (for instance, see \cite{MR4197372,MR1220599,MR3648208,MR4699495}), not many structural theorems for WTD graphs are known.
The study of WTD graphs began with the work of Bert and Douglas \cite{MR1605080}, where WTD cycles and paths are characterized and several constructions of WTD trees are given.
In \cite{finbow2009total}, the authors study WTD graphs via the composition and decomposition of graphs, showing that every WTD tree can be constructed from a family of three small trees.
Later, in \cite{MR4333882}, the authors show that any WTD graph with a bounded total domination number can be recognized in polynomial time, and focus on WTD graphs with a total domination number of $2$.
To our knowledge, these are the only papers containing structural results of WTD graphs.

In this paper, we give two characterizations of WTD trees.
Our characterizations are different from the one in \cite{finbow2009total}.
The characterization in \cite{finbow2009total} finds a path in a given WTD tree with a special property, which is used to show which substructure is can appear in the decomposition process.
On the other hand, we detect the WTD-ness of a tree by 2-coloring the tree (with red and blue); see Theorem~\ref{thm. char. of WTD trees}.

Essentially, $T$ has two noteworthy ``interior'' subgraphs, $T_\R$ and $T_\B$, which are induced subforests that arise from a red-blue coloring of $T$. 
We show that $T$ is WTD if and only if both $T_\R$ and $T_\B$ are WTD (see Corollary~\ref{cor. T WTD <-> TBTR WTD}).
The connected components of $T_\R$ and $T_\B$ have a property that we call ``balanced.''
The general properties of balanced trees are the focus of Section~\ref{sec. Balanced trees}, and their WTD-ness is the topic of Section~\ref{sec. des. char. of WTD bal. trees}. 
The main result is our descriptive characterization of them in Theorem~\ref{thm. char. of WTD delt. trees}.
The second characterization of WTD trees is constructive, allowing us to build WTD balanced trees.
It is the main result of Section~\ref{sec. const. WTD bal. trees}; see Theorem~\ref{thm. constructing hh = 3 delt. trees}.
In Section~\ref{sec. dom selector}, we define a domination selector, which is used to establish the minimality of a TDS in the proofs of subsequent sections.

\begin{assumptions}
Throughout, we assume that $G$ is a finite, simple graph.
If the graph we consider is clear in context, then we write $V := V(G)$ and $E := E(G)$.
Since $G$ is simple, we denote each edge $e \in E$ as $uv=vu$ where $u,v \in V$ are the distinct vertices incident to $e$.
\end{assumptions}

\noindent\textbf{Acknowledgement.} We are gratful to Clemson University's School of Mathematical and Statistical Sciences for financial support for the 2020 summer REU COURAGE (Clemson Online Research on Algebra and Graphs Expanded) where this research started.

\section{Domination selector}\label{sec. dom selector}

The goal of this section is to define the notion of a \emph{domination selector}, which is a useful tool for demonstrating the minimality of a given set of vertices. 
The main result of this section is Lemma~\ref{lem. D(v) not in intersection -> union is minimal}, which we use subsequently in several key places.
We begin with some notations and definitions.

\begin{notation}
    Recall that a forest is a finite disjoint union of trees.
    Let $x,y \in V$.
    A \emph{path} from $x$ to $y$ is a sequence of distinct vertices $(v_0,\dots,v_d)$ where $d \in \mathbb{N}$, $v_0 = x$, $v_d = y$, and $v_iv_{i+1} \in E$ for all $i = 0,\dots,d-1$.
    The \emph{length} of a path $(v_0,\dots,v_d)$ is $d$ also written as $d =: |(v_0,\dots,v_d)|$ , and the \emph{distance} from $x$ to $y$ is the length of the shortest path from $x$ to $y$, denoted $\dist(x,y)$.
    In case the given graph is a tree (or a forest), denote the (shortest) path from $x$ to $y$ by $\pth(x,y)$.
    If a vertex $z \in V$ equals to one of the $v_j$ in $\mbox{path}(x,y)$, we write $z \in \pth(x,y)$.
\end{notation}

\begin{definition}
A vertex $v \in V$ is \emph{isolated} if $\deg(v) = 0$, it is a \emph{leaf} if $\deg(v) = 1$, and it is a \emph{support vertex} if it is adjacent to a leaf.
Every vertex adjacent to a support vertex is called a \emph{supported vertex}.
\end{definition}

\begin{definition}
Let $G$ be a graph with at least one leaf.
The \emph{height} of a vertex $v \in V$ is given by
$$
\hh(v) := \min\Set{\dist(v,\ell)}{\ell \mbox{ is a leaf in } G}.
$$ 
If $v$ is an isolated vertex, then we set $\hh(v) = 0$ as a convention. 
We denote 
$$
V_k = V_k(G) := \{v \in V\ |\  \hh(v) = k\}
$$ 
and the \emph{height} of $G$ is the integer 
$$
\hh(G) := \max\Set{k \in \mathbb{N}}{V_k \neq \emptyset}.
$$
For instance, every leaf of a graph has height 0, and every non-leaf support vertex has height~1.
\end{definition}

\begin{lemma}
\label{lem. singleLeaf in a min TD set}
Let $G$ be a graph with a support vertex $s$.
Then for any minimal TDS $D$, we have $|D \cap N(s) \cap V_0(G)| \leq 1$.
\end{lemma}

\begin{proof}
Assume not. 
Then there are two vertices $x,y \in D \cap N(s) \cap V_0(G)$.
Then $N(D \setminus \{x\}) = V(G)$ as $s \in N(y)$, which contradicts the minimality of $D$. 
\end{proof}

Lemma~\ref{lem. LeafAdding} is the first application of Lemma~\ref{lem. singleLeaf in a min TD set}.
It states that given a graph $G$, adding a leaf to a support vertex or deleting a leaf adjacent to a support vertex with more than one leaf adjacent  does not change the WTD-ness of $G$.

\begin{lemma}
\label{lem. LeafAdding}
Let $G$ be a connected graph with at least one leaf.
Let $G'$ be a graph obtained from $G$ by attaching any number of leaves to its suppport vertices.
Then $G$ is WTD if and only if $G'$ is WTD.
\end{lemma}

\begin{proof}
Let $s \in V(G)$ be a support vertex adjacent to a leaf $\ell \in V(G)$.
Let $H$ be a graph with $V(H) = V(G) \cup \set{\ell'}$ and $E(H) = E(G) \cup \set{s\ell'}$.
It suffices to prove the claim for $G$ and $H$.

Let $S \subseteq V(G)$ be a minimal TDS of $G$.
Since $N_G(S) = V(G) \ni \ell$, we have $s \in S$ as $s$ is the only vertex adjacent to $\ell$.
Now consider $S\subseteq H$.
Since $s \in S$, we have $N_H(S) = V(H)$.
To show that $S$ is a minimal TDS of $H$, assume by way of contradiction that there is a proper subset $S' \subsetneq S$ with $N_H(S') = V(H)$.
Then the set $S' \subseteq V(G)$ has $N_G(S') = V(G)$, contradicting the minimality of $S$ in $G$.
Hence $S$ is a minimal TDS of $H$.

Next, let $S \subseteq V(H)$ be a minimal TDS of $H$.
We show that there is a minimal TDS of $G$ of the same size as $S$.
If $\ell' \not\in S$, then $S \subseteq V(G)$ is also a minimal TDS of $G$ by the preceding paragraph.
So suppose that $\ell' \in S$.
Then by Lemma~\ref{lem. singleLeaf in a min TD set}, $\ell'$ is the only leaf in $S$ adjacent to $s$.
By the minimality of $S$, we have $N_H(S\setminus \set{\ell'}) = V(H)\setminus \set{s}$.
Hence the set $S' := (S \setminus \set{\ell'}) \cup \set{\ell}$ is a minimal TDS of $H$.
Since $S' \subseteq V(G)$, the set $S' \subseteq V(G)$ has $N_G(S') = V(G)$.
To show that $S'$ is minimal in $G$, asumme by way of contradiction that there is a proper subset $S'' \subsetneq S'$ with $N_G(S'') = V(G)$.
Then by the previous argument, $S''$ is a minimal TDS of $H$ contained in $S'$, contradicting the minimality of $S'$ in $H$.
Thus $S'$ is a minimal TDS of $G$ with $|S'| = |S|$.

So far, we have shown that any minimal TDS of $G$ is a minimal TDS of $H$, and for any minimal TDS of $H$, there is a minimal TDS of $G$ of the same size.
This implies that $G$ is WTD if and only if $H$ is WTD.
\end{proof}

We next extend the definition of minimality for an arbitrary set with respect to open neighborhoods.

\begin{definition}\label{def. minimal}
A subset $D \subseteq V(G)$ is \emph{minimal} (with respect to open neighborhoods) if there is no proper subset $D' \subsetneq D$, such that $N(D') = N(D)$.
\end{definition}

From now on, a set being minimal will always refer to Definition~\ref{def. minimal}.
The following lemma gives an equivalent definition of a minimal set.

\begin{lemma}
\label{lem. minimalEquiv}
Let $D \subseteq V$. Then $D$ is minimal if and only if for every vertex $v \in D$, there is a vertex $u \in N(D)$ such that  $N(u) \cap D = \{v\}$.
\end{lemma}

\begin{proof}
First, suppose that $D$ is minimal.
By way of contradiction, assume that there exists a vertex $c \in D$ such that for all $x \in N(c)$ we have $(N(x) \cap D) \setminus \{c\} \neq \emptyset$.
We claim that $N(D\setminus \set{c}) = N(D)$, which contradicts the minimality of $D$.
The containment $N(D\setminus \set{c}) \subseteq N(D)$ is by definition. 
For the other containment, let $v \in N(D)$.
Then there exists $u \in D$ such that $vu \in E$.
If $u \neq c$, then $u \in D\setminus \set{c}$, hence $v \in N(D \setminus \set{c})$.
So suppose that $u = c$.
Then we have $(N(v) \cap D)\setminus \set{c} \neq \emptyset$.
So, there exists $c' \in D$ with $vc' \in E$ and $c' \neq c$.
Thus $v \in N(D \setminus \set{c})$.

Now suppose that for every vertex $v \in D$, there is a vertex $u \in N(D)$ such that $N(u) \cap D = \{v\}$.
Assume for a contradiction that $D$ is not minimal.
Hence there is $c \in D$ such that $N(D \setminus \{c\}) = N(D)$.
By hypothesis, there is some vertex $y \in N(D)$ such that $N(y) \cap D = \{c\}$.
Hence $N(D\setminus \{c\}) \not\ni y$, a contradiction.
\end{proof}

Using Lemma~\ref{lem. minimalEquiv}, we can now define domination selectors.

\begin{definition}\label{def. dom. selector}
Let $D \subseteq V(G)$ be a minimal set. 
Let $\mathcal{D}:D\to N(D)$ be a function such that for all $v \in D$, we have
$N(\mathcal{D}(v)) \cap D = \set{v}$.
We call such a function a \emph{domination selector} of $D$ in $G$.
\end{definition} 

Domination selectors are well defined, injective functions by Lemma \ref{lem. minimalEquiv}.
Note that domination selectors of a minimal set might not be unique.
The following is a corollary of Lemma~\ref{lem. minimalEquiv} and Definition~\ref{def. dom. selector}.

\begin{corollary}\label{cor. dom selector <-> minimal}
Let $D \subseteq V$ such that $N(D) = V$. 
Then $D$ is a minimal TDS if and only if $D$ has a domination selector.
\end{corollary}

\begin{example}
\label{domination selector example}
Consider the following graph $Y$.
%(Figure~\ref{Yifan Graph}).
\begin{figure}[ht]
\centering
\includegraphics{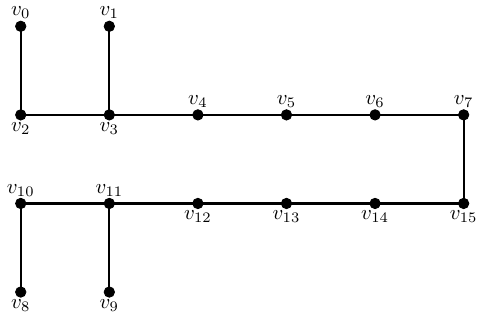}
%\caption{Graph $Y$, the Yifan graph}
\label{Yifan Graph}
\end{figure}

\noindent One can check that $D := \{v_2, v_3, v_5, v_6, v_{10}, v_{11}, v_{14}, v_{15}
\}$ is a TDS in $Y$, and that the function $\mathcal{D}: D \rightarrow V(Y)$ given by
$$
\mathcal{D}(x) = \begin{cases}
v_0, & x = v_2\\
v_1, & x = v_3\\
v_6, & x = v_5\\
v_5, & x = v_6\\
v_8, & x = v_{10}\\
v_9, & x = v_{11}\\
v_{15}, & x = v_{14}\\
v_{14}, & x = v_{15}\ .\\
\end{cases}
$$
is a domination selector of $D$ in $Y$.
Hence by Corollary~\ref{cor. dom selector <-> minimal}, $D$ is minimal.
\end{example}

Another way of constructing a minimal set is described in the following two results.

\begin{lemma}\label{lem. D(v) not in intersection -> union is minimal}
    Let $D', D'' \subseteq V$ be disjoint minimal sets, and let $I = N(D') \cap N(D'')$.
    If there are domination selectors $\mathcal{D}':D' \to N(D)$ and $\mathcal{D}'': D'' \to N(D'')$ such that $\mathcal{D}'(D') \cap I = \emptyset$ and $\mathcal{D}''(D'') \cap I = \emptyset$, then $D = D' \cup D''$ is also minimal.
\end{lemma}

\begin{proof}
    Define $\mathcal{D}:D \to N(D)$ by
    $$
    \mathcal{D}(x) := \begin{cases}
        \mathcal{D}'(x) & x \in D'\\
        \mathcal{D}''(x) & x \in D''\ .
    \end{cases}
    $$
    We show that $\mathcal{D}$ is a domination selector of $D$.
    Let $v \in D$.
    By symmetry of $D'$ and $D''$, suppose that $v \in D'$.
    Then we have
    $$
    N(\mathcal{D}(v)) \cap D = N(\mathcal{D}'(v)) \cap D.
    $$
    Since $\mathcal{D}'(v) \not\in I$, we have $N(\mathcal{D}'(v)) \cap D'' = \emptyset$.
    Hence we get
    $$
    N(\mathcal{D}'(v)) \cap D = (N(\mathcal{D}'(v)) \cap D') \cup (N(\mathcal{D}'(v)) \cap D'') = \set{v} \cup \emptyset = \set{v}.
    $$
    Therefore, $\mathcal{D}$ is a domination selector of $D$.
\end{proof}

\begin{corollary}\label{cor. neighbor disjoint min. sets  -> union is min.}    
    Let $D', D'' \subseteq V$ be disjoint minimal sets such that $N(D') \cap N(D'') = \emptyset$.
    Then $D = D' \cup D''$ is also a minimal set.
\end{corollary}

\section{Balanced trees, RD-, and BDSs}\label{sec. Balanced trees}

In this section, we define balanced trees and two types of dominating sets which we then use to give our first characterization of WTD trees in the main result of this section, Corollary~\ref{cor. WTD iff RDBD WTD}. 

\begin{assumptions}
In this section, we assume that $T = (V,E)$ is a tree with a 2-coloring $\chi: V \to \set{\R,\B}$ where $\R$ and $\B$ stands for red and blue color, respectively.
\end{assumptions}

\begin{definition}\label{def. delt. tree}
    The tree $T$ is \emph{balanced} if no two vertices of the same height are adjacent.
\end{definition}

Proposition~\ref{prop. equiv. delt. tree} gives an equivalent definition of balanced trees.
We next give a lemma used in its proof.

\begin{lemma}\label{lem. adjacent, then height differ by 1}
    Let $T$ be a balanced tree, and let $u,v \in V$ be distinct vertices that are adjacent to each other.
    Then $\hh(u) - \hh(v) = \pm 1$.
\end{lemma}

\begin{proof}
    Set $\hh(u) = h$ and $\hh(v) = h'$.
    Since $T$ is a balanced tree, we must have $h \neq h'$ as $uv \in E$.
    Without loss of generality, assume that $h > h'$.
    We claim that $h' = h - 1$.
    By way of contradiction, suppose that $h' < h - 1$.
    Then there exists a leaf $v_0 \in V_0$ so that $\pth(v_0,v) = (v_0,\dots,v_{h'})$ with $v_{h'} = v$.
    Now consider $\pth(v_0,u) = (v_0,\dots,v_{h'},u)$; the equality holds by the uniqueness of paths in trees.
    Since we have
    \begin{align*}
       |\pth(v_0,u)| = h'+1 < h 
    \end{align*}
    $\hh(u)$ is less than $h$, a contradiction.
\end{proof}

\begin{corollary}\label{cor. same hh -> odd dist}
    Let $T$ be a balanced tree and let $u,v \in V_h$ for some $0 \leq h \leq \hh(T)$.
    Then $\dist(u,v)$ is even.
\end{corollary}

\begin{proof}
    If $u = v$, then $\dist(u,v) = 0$ which is even.
    So assume that $u \neq v$.
    By way of contradiction, assume that $\dist(u,v) = 2k + 1$ for some $k \in \mathbb{N}$.
    Then by inductively applying Lemma~\ref{lem. adjacent, then height differ by 1} on $k$, one can show that the parity of $\hh(u)$ and $\hh(v)$ are different, hence $\hh(u) \neq \hh(v)$, a contradiction. 
\end{proof}

\begin{proposition}\label{prop. equiv. delt. tree}
The following conditions on a tree $T$ are equivalent:
\begin{enumerate}
	\item $T$ is balanced.
	\item Any two vertices of the same height have the same color under $\chi$.
	\item Every leaf has the same color.
\end{enumerate}
\end{proposition}

\begin{proof}
\noindent $(1 \Rightarrow 2)$ Assume $T$ is balanced, and let $u,v \in V_h$ for some $0 \leq h \leq \hh(T)$.
Then by Corollary~\ref{cor. same hh -> odd dist}, $\dist(u,v)$ is even, hence $\chi(u) = \chi(v)$.

$(2 \Rightarrow 3)$ Since all leaves have height 0, they have the same color.

$(3\hspace{-1.5mm} \Rightarrow \hspace{-1mm} 1)$ Assume that every leaf in $T$ has the same color. 
Let $u,v \in V_h$ for some $0 \leq h \leq \hh(T)$.
If $h = 0$, then $u$ and $v$ are leaves, so they have the same color by assumption, implying they are not adjacent.
If $h > 0$, then there are leaves $u_0,v_0 \in V_0$ so that $\pth(u_0,u) = (u_0,\dots,u_h)$ and $\pth(v_0,v) = (v_0,\dots,v_h)$ where $u_h = u$ and $v_h = v$.
Since $u_0$ and $v_0$ have the same color, $u$ and $v$ must have the same color, implying that $uv \not\in E$.
\end{proof}

\begin{example}
    Consider the following trees with 2-coloring.
    \begin{figure}[H]
        \centering
        \includegraphics{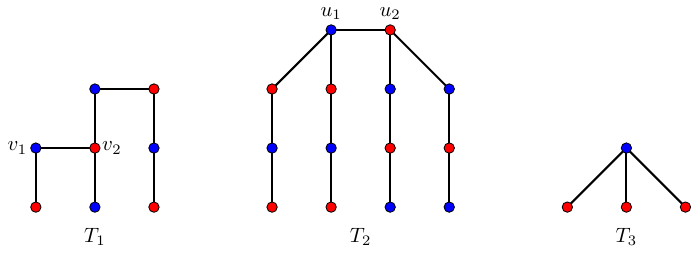}
        \caption{Trees with 2-colorings}
        \label{fig. 3 trees with coloring}
    \end{figure}        
    \noindent The tree $T_1$ is not balanced since $\hh(v_1) = 1 = \hh(v_2)$ but $v_1v_2 \in E(T_1)$.
    Similarly, $u_1$ and $u_2$ show that $T_2$ is not balanced.
    We can also deduce this from Proposition~\ref{prop. equiv. delt. tree} since the leaves in each tree do not have the same color.
    On the other hand, $T_3$ is a balanced tree by definition, or since all of its leaves share the same color.
\end{example}

The following corollary is a quick consequence of Proposition~\ref{prop. equiv. delt. tree}.

\begin{corollary}\label{cor. same height parity -> same color}
    Let $T$ be a balanced tree.
    Then all vertices of even height have the same color, and all vertices of odd height have the same color.
\end{corollary}

Now we give a way to partition a TDS of a tree into 2 new types of dominating sets.
This is used to construct two minimal TDS with different sizes.

\begin{notation}
    Denote the set of red and blue vertices of $T$ by $V_{\R}$ and $V_{\B}$, respectively; i.e., for all $C \in \set{\R,\B}$, define $V_C := \chi^{-1}(C)$.
\end{notation}
    
\begin{definition}
    Let $D \subseteq V$.
    Then $D$ is a \emph{red-dominating set} (RDS) if $N(D) = V_\R$. 
    And $D$ is a \emph{blue-dominating set} (BDS) if $N(D) = V_\B$.
\end{definition}

Note that by the definition of 2-coloring, if $D$ is an RDS, then $D \subseteq V_\B$, and $D \subseteq V_\R$ if $D$ is a BDS.
Hence we can state the following as a fact.

\begin{fact}\label{fact. TDS iff RDS U BDS}
    Let $D' \subseteq V_\B$, $D'' \subseteq V_\R$, and $D = D' \cup D''$. 
    Then
    \begin{itemize}
        \item[(1)] $D$ is a TDS if and only if $D'$ and $D''$ are red-dominating and blue-dominating, respectively.
        \item[(2)] $D$ is a minimal TDS if and only if $D'$ and $D''$ are minimal red-dominating and minimal blue-dominating, respectively.
    \end{itemize}
\end{fact}

\begin{example}
Consider the tree $Y$ with vertices colored red and blue in Figure~\ref{fig. Yifan Graph colored}.
\begin{figure}[ht]
\centering
\includegraphics{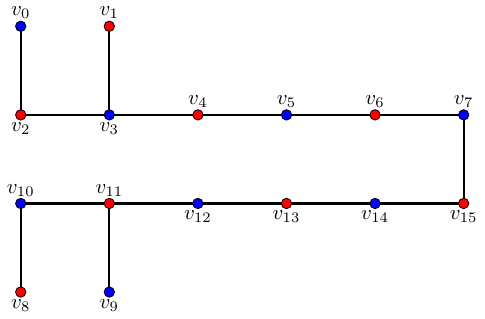}
\caption{Graph $Y$ with a 2-coloring}
\label{fig. Yifan Graph colored}
\end{figure}
The set $D := \{\tcr{v_2}, \tcb{v_3}, \tcb{v_5}, \tcr{v_6}, \tcb{v_{10}}, \tcr{v_{11}}, \tcb{v_{14}}, \tcr{v_{15}}\}$ is a minimal TDS of $Y$ by Example \ref{domination selector example}.
Now, we partition the set $D$ by the coloring: set  $D' := \{ \tcb{v_3}, \tcb{v_5},  \tcb{v_{10}},  \tcb{v_{14}}\}$ and $D'' := \{\tcr{v_2},\tcr{v_6},\tcr{v_{11}},\tcr{v_{15}}\}$.
Then it is straightforward to show that $D'$ and $D''$ are red-dominating and blue dominating, respectively.
Now, define $\mathcal{D}': D' \rightarrow V(Y)$ and $\mathcal{D}'': D'' \rightarrow V(Y)$ by
\begin{align*}
&\mathcal{D}'(x) = \begin{cases}
v_1, & x = v_3\\
v_6 ,&  x = v_5\\
v_8 ,& x = v_{10}\\
v_{15} ,& x = v_{14}\\
\end{cases}
 &\mbox{ and }&
&\mathcal{D}''(x) = \begin{cases}
v_0, &  x = v_2\\
v_5, &  x = v_6\\
v_9, &  x = v_{11}\\
v_{14}, &  x = v_{15}\\
\end{cases}.
\end{align*}
Notice that $\mathcal{D}'$ and $\mathcal{D}''$ are obtained from the function $\mathcal{D}$ from Example \ref{domination selector example} by restricting its domain to $D'$ and $D''$, respectively.
Hence the functions $\mathcal{D}'$ and $\mathcal{D}''$ are domination selectors for $D'$ and $D''$ respectively, hence both sets are minimal.
\end{example}

\begin{corollary}\label{cor. WTD iff RDBD WTD}
    A tree $T$ is WTD if and only if every minimal RDS has the same size and every minimal BDS has the same size.
\end{corollary}

\begin{proof}
    Apply Fact~\ref{fact. TDS iff RDS U BDS} with $D' = D \cap V_\B$ and $D'' = D \cap V_\R$.
\end{proof}

\section{Descriptive characterization of WTD balanced trees}\label{sec. des. char. of WTD bal. trees}

Theorem~\ref{thm. char. of WTD delt. trees} is the main result of this section.
It allows us to detect when a balanced tree is WTD.
The case of an arbitrary tree is handled in Section~\ref{sec. interior graphs of trees} below.
We begin with a definition.

\begin{definition}\label{def. radar}
    Let $x \in V(G)$.
    The \emph{radar of $x$ of distance $d$} (in $G$) is the set of vertices
    $$
    R(x,d) = R_G(x,d) := \Set{y \in V}{\dist(x,y) = d}.
    $$
    Now let $T$ be a tree with $x,r \in V(T)$.
    The \emph{branch spanned from $r$ to $x$} is the set of vertices
    $$
    \br_r(x) := \Set{y \in V(T)}{x \in \pth(y,r)}.
    $$
\end{definition}

\begin{example}
\label{branch and radar example}
Let $T$ be the tree given in Figure~\ref{fig. branch and radar}.
\begin{figure}[ht]
\centering
\includegraphics[scale = 0.85]{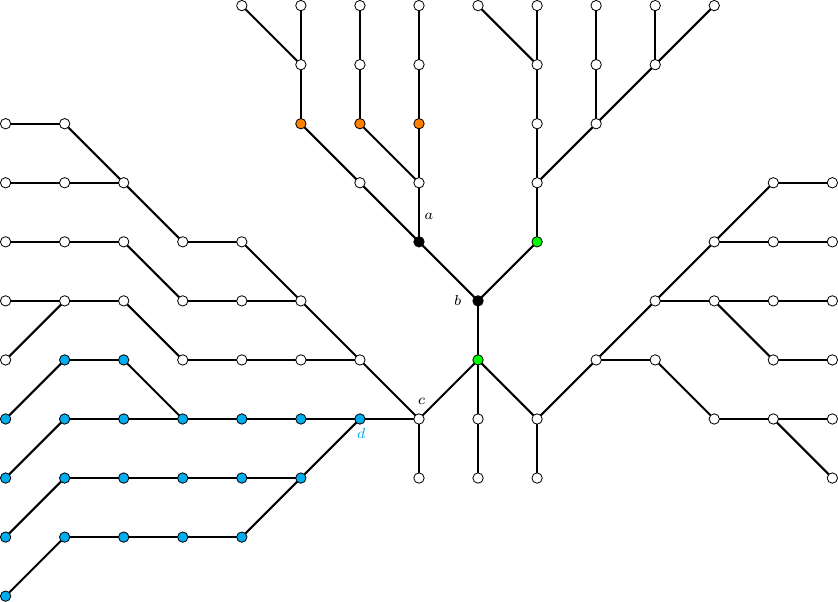}
\caption{Branch and radar example on $G$}
\label{fig. branch and radar}
\end{figure}
Then the radar $R(a,2)$ is the set of \textcolor{green}{green} and \textcolor{orange}{orange} vertices.
The branch $\br_c(d)$ is the set of all \textcolor{cyan}{cyan} vertices; note that $c \not\in \br_c(d)$.
Let $H$ be the subgraph induced by the vertex set $\br_b(a)$.
Then the radar $R_H(a,2)$ is the set of \textcolor{orange}{orange} vertices.
\end{example}

The next few results provide conditions that force a balanced tree to be not WTD by constructing two RDSs of different sizes and using Corollary~\ref{cor. WTD iff RDBD WTD}.

\begin{assumptions}
    For the remainder of this section, given a balanced tree assume that the even height vertices are blue and odd height vertices are red using Corollary~\ref{cor. same height parity -> same color}.
\end{assumptions}

\begin{lemma}\label{lem. heighestSmallest}
Let $T$ be a balanced tree of height $d > 0$.
Then $|V_{d-1}| > |V_{d}|$.
\end{lemma}

\begin{proof}
By way of contradiction, assume that $|V_{d-1}| \leq |V_{d}|$.
Since no vertex in $V_d(T)$ is a leaf, every vertex in $V_d(T)$ is adjacent to at least 2 vertices in $V_{d-1}(T)$. 
Hence there are at least $2|V_d|$ edges between $V_d$ and $V_{d-1}$. 
Thus the subgraph induced by $V_d(T) \cup V_{d-1}(T)$ is not a forest; any forest $H$ with $k$ connected components has exactly $|V(H)| - k$ number of edges, but the subgraph induced by $V_d \cup V_{d-1}$ has more edges than the number of vertices as $2 |V_d| \geq |V_d| + |V_{d-1}|$.
Since an induced subgraph of $T$ is not a forest, $T$ is not a tree, a contradiction.
\end{proof}

The next result extends the WTD-ness of $P_4$.

\begin{theorem}\label{thm. P5 -> Mix}
Let $T$ be a balanced tree. If $T$ has two leaves $l_1$ and $l_2$ such that $\dist(l_1,l_2) = 4$, then $T$ is not WTD. 
\end{theorem}

\begin{proof}
Let $s_1$ and $s_2$ be the support vertices adjacent to $l_1$ and $l_2$ respectively, and let $v \in V_2$ be the vertex which is adjacent to both $s_1$ and $s_2$.
By Corollary~\ref{cor. WTD iff RDBD WTD}, it suffices to show that there are two minimal RDSs of different sizes.
By Lemma~\ref{lem. LeafAdding}, we assume that there is a unique leaf for each support vertex in $T$.
Let $u_1,u_2,\cdots,u_k$ be vertices which are adjacent to $v$ which are not support vertices (hence $\hh(u_i) = 3$ by Lemma~\ref{lem. adjacent, then height differ by 1} for all $i$).
Let $s_1,s_2,\cdots,s_m$ be the support vertices adjacent to $v$, and let $l_i$ be the leaf adjacent to $s_i$ for each $i$.
Note that $u_1,\cdots,u_k \in V_\R$ as $v \in V_\B$ since $T$ is a balanced tree (see Figure~\ref{Figure1}).

\begin{figure}[ht]
\centering
\includegraphics[scale = 0.91]{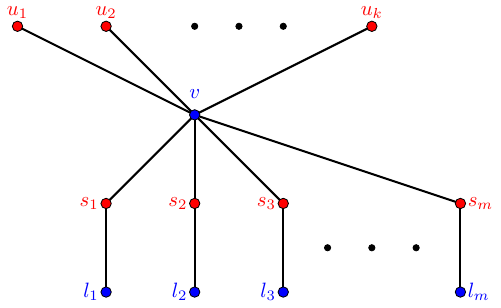}
\caption{Tree $T$ with vertices labeled and colored}
\label{Figure1}
\end{figure}

For $i=1,2,\dots,k$, define $G_{u_i}$ to be the subgraph of $G$ induced by $\br_v(u_i) =: B_{u_i}$.
For each $G_{u_i}$, consider the union of radars $R_{u_i} := R_{G_{u_i}}(u_i,1) \cup \left(\bigcup_{j = 1}^{\infty} R_{G_{u_i}}(u_i,3 + 2j)\right)$.
Since $u_i$ is red color, the set $R_{u_i}$ is colored blue by Lemma~\ref{lem. adjacent, then height differ by 1} and Corollary~\ref{cor. same height parity -> same color}.
Also, one can verify that $R_{u_i}$ is an RDS in $G_{u_i}$.
Since $R_{u_i}$ is an RDS in $G_{u_i}$, there is some minimal RDS $D_{u_i} \subseteq R_{u_i}$ in $G_{u_i}$.
Set $D_u := \cup_{i = 1}^k D_{u_i}$.
By Corollary~\ref{cor. neighbor disjoint min. sets  -> union is min.}, $D_u$ is minimal.

Next, for $j= 1,2,\dots,m$, let $G_{s_j}$ be the subgraph of $G$ induced by 
$B_{s_j}:=\br_v(s_j) \setminus \{l_j\}$.
For each $G_{s_j}$, define the set $R_{s_j} = \bigcup_{i = 0}^{\infty} R_{G_{s_j}}(s_j,3 + 2i)$ (see Figure~\ref{Figure2}).
Then $R_{s_j} \subseteq V_\B(T)$, and $R_{s_j}$ dominates all red vertices in $G_{s_j}$ but $s_j$.
Let $D_{s_j} \subseteq R_{s_j}$ be a minimal set of $R_{s_j}$, and set $D_s := \bigcup_{j = 1}^{m}D_{s_j}$.
Again by Corollary~\ref{cor. neighbor disjoint min. sets  -> union is min.}, $D_s$ is minimal.

\begin{figure}[ht]
\centering
\includegraphics{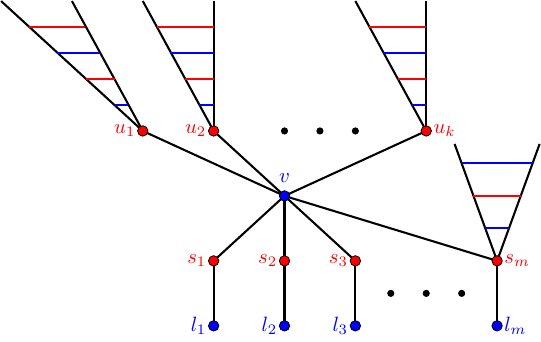}
\caption{Tree $T$ with branches (there are branches growing from all $s_i$'s like $s_m$)}
\label{Figure2}
(View the horizontal lines as the set of vertices)
\end{figure}

Let $D = D_u \cup D_s \cup \{v\}$ and let $D' = D_u \cup D_s \cup \{l_1,l_2,\cdots,l_m\}$.
We claim that $D$ and $D'$ are minimal RDSs in $G$.
Since $N(D_u \cup D_s) = V_\R(T) \setminus \{s_1,s_2,\cdots,s_k\}$, we have $N(D) = N(D') = V_\R(T)$ as $\set{s_1,\dots,s_m} \subseteq N(v)$ and $\set{s_1,\dots,s_m} \subseteq N(l_1,\dots,l_m)$.

Now we check the minimality of $D$ and $D'$.
Set $D_{u,s} := D_u \cup D_s$.
Then $D_{u,s}$ is minimal by Corollary~\ref{cor. neighbor disjoint min. sets  -> union is min.}, hence there exists a domination selector $\mathcal{D}_{u,s}: D_{u,s} \to N(D_{u,s})$.
Note that for each $i$ and for all $x \in D_{u_i}$, we may choose $\mathcal{D}_{u,s}(x) \in B_{u_i}\setminus\set{u_i}$ since $R_{G_{u_i}}(u_i,3) \cap D_{u_i} = \emptyset$.
Also, for all $x \in D_{s_j}$ for each $j$, we have $\mathcal{D}_{u,s}(x) \neq s_j$ since $s_j \not\in N_{G_{s_j}}(D_{s_j})$.

As $\set{v}$ is minimal with domination selector $\mathcal{D}_v: \set{v} \to N(v)$ given by $\mathcal{D}_v(v) = s_1$,
setting 
$$
I := N(D_{u,s}) \cap  N(v) = \set{u_1,\dots,u_k}
$$
we have $\mathcal{D}_{u,s}(D_{u,s}) \cap I = \emptyset$.
So, $D$ is minimal by Lemma~\ref{lem. D(v) not in intersection -> union is minimal}.

Since $N(D_{u,s}) \cap N(s_1,\dots,s_m) = \emptyset$, $D'$ is minimal by Corollary~\ref{cor. neighbor disjoint min. sets  -> union is min.}.
Since $m > 1$, $D$ and $D'$ are two minimal RDSs of different sizes, so $T$ is not WTD, as desired.
\end{proof}

%!%

\begin{example}\label{ex. dist 4 leaves}
We demonstrate the construction of two minimal RDSs of different sizes shown in Theorem~\ref{thm. P5 -> Mix}.
Consider the balanced tree in Figure~\ref{distance 4 leaves figure}. 
\begin{figure}[ht]
\centering
\includegraphics{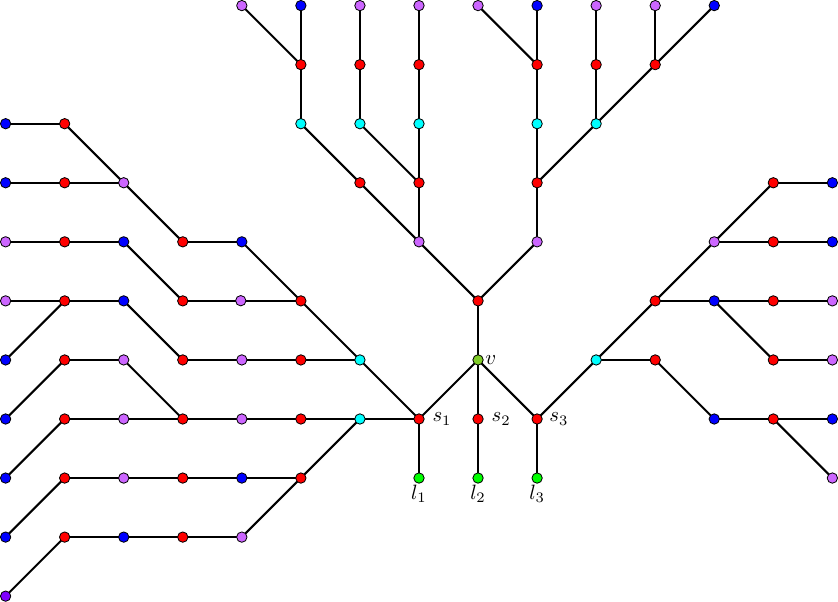}
\caption{Balanced tree with distance 4 apart leaves}
\label{distance 4 leaves figure}
\end{figure}
We use the notation from the proof of Theorem~\ref{thm. P5 -> Mix}.
The \textcolor{cyan}{cyan} vertices represent the even height vertices (blue vertices) that were excluded when we constructed $R_{u_i}$ and $R_{s_i}$.
Let $P$ be the set of all \textcolor{ccwwff}{purple} vertices ($P$ corresponds to the set $D_u \cup D_s$).
Notice that $P$ is minimal, and that $N(P) = V_\R \setminus \{s_1,s_2,s_3\}$.
One can compute that the sets $D := P \cup \set{v}$ and $D' := P \cup \set{l_1,l_2,l_3}$ are minimal RDSs.

Also, one can choose a different purple set (minimal, dominates all red vertices but the $s_i$'s, and does not include any \textcolor{cyan}{cyan} vertices) and see that the union of the new purple set with either $\set{v}$ or $\set{l_1,l_2,l_3}$ forms a minimal RDS; minimality of the union is preserved by not picking the \textcolor{cyan}{cyan} vertices.
\end{example}

The WTD-ness of $P_4$ makes another appearance in our next result.

\begin{lemma}\label{lem. no 4-path iff |N(v) cap V_1| = 1}
    Let $T$ be a balanced tree.
    Then for all leaves $\ell_1,\ell_2 \in V_0$, we have $\dist(\ell_1,\ell_2) \neq 4$ if and only if $|N(v) \cap V_1| = 1$ for all $v \in V_2$.
\end{lemma}

\begin{proof}
    ($\Rightarrow$) Assume that $\dist(\ell_1,\ell_2) \neq 4$ for all leaves $\ell_1$, $\ell_2$, and suppose by way of contradiction that there is a vertex $v \in V_2$ such that $|N(v) \cap V_1| \neq 1$.
    By definition of height, we must have $|N_T(v) \cap V_1| \geq 1$.
    So, there are distinct $s_1,s_2 \in N(v) \cap V_1$.
    Then there are leaves $\ell_1,\ell_2 \in V_0$ such that $\ell_1s_1,\ell_2s_2 \in E$.
    This forms a path between $\ell_1$ and $\ell_2$, namely $\ell_1s_1vs_2\ell_2$, so $\dist(\ell_1,\ell_2) \leq 4$, hence $\dist(\ell_1,\ell_2) < 4$ by assumption.
    Since $s_i$ is the unique neighbor of $\ell_i$ for $i = 1,2$, it follows that $s_1$ and $s_2$ are adjacent, contradicting the fact that $T$ is balanced.

    ($\Leftarrow$) Assume that $|N(v) \cap V_1| = 1$ for all $v \in V_2$. 
    Let $\ell_1,\ell_2 \in V_0$ be two leaves.
    Then there are support vertices $s_1,s_2 \in V_1$ such that $\ell_1s_1,\ell_2s_2 \in E$.
    If $s_1 = s_2$, then $\ell_1 s_1\ell_2$ forms a path of length 2 between $\ell_1$ and $\ell_2$.
    Thus by the uniqueness of path between two vertices in a tree, we get $\dist(\ell_1,\ell_2) = 2 \neq 4$.
    Now suppose that $s_1 \neq s_2$.
    Since $T$ is a balanced tree, there is no edge between $s_1$ and $s_2$.
    Thus we must have $\dist(\ell_1, \ell_2) > 3$.
    Assume by way of contradiction that $\dist(\ell_1,\ell_2) = 4$. 
    Then there exists $v \in V_2$ such that $\ell_1 s_1 v s_2 \ell_2$ forms a path of length 4.
    This implies that $s_1,s_2 \in N(v)$, so $|N(v) \cap V_1| \geq 2$, contradicting the assumption that $|N(v) \cap V_1| = 1$ as $v \in V_2$.
    Thus $\dist(\ell_1,\ell_2) \neq 4$.
\end{proof}

Toward our descriptive characterization of WTD balanced trees we next show that all balanced trees of height at least 4 are not WTD.

\begin{theorem}
\label{thm. height >= 4 is not WTD}
Let $T$ be a balanced tree with blue leaves and $\hh(T) \geq 4$. 
Then $T$ has two minimal RDSs of different sizes, and $T$ is not WTD.
\end{theorem}

\begin{proof}
Suppose by way of contradiction that all minimal RDSs of $T$ have the same size.
Since $\hei(T) \geq 4$, there is some vertex $v$ of height $4$.
Hence there are vertices $w_0,x,s,\ell$ where $\ell$ is a leaf such that $\pth(v,\ell)= (v,w_0,x,s,\ell)$.
By Corollary~\ref{cor. WTD iff RDBD WTD}, Theorem~\ref{thm. P5 -> Mix}, and Lemma~\ref{lem. no 4-path iff |N(v) cap V_1| = 1}, for all $a \in V_2(T)$, we have $|N(a) \cap V_1(T)| = 1$.
 Let $w_0,w_1,w_2,\cdots,w_k$ be the distinct neighbors of $v$.
Choose $v$ to be the root of $T$.
For each $i= 1,\dots, k$, let $G_i$ be the subgraph induced by $\br_v(w_i) =: B_i$, and let $G_0$ to be the subgraph induced by $B_0 := \br_v(w_0)\setminus \{x,s ,\ell\}$.

\begin{figure}[ht]
\centering
\includegraphics[scale = 1.25]{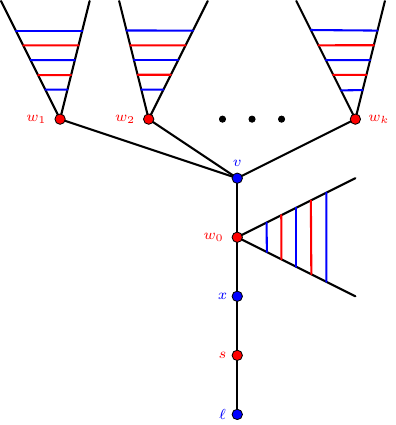}
\caption{Tree $T$ with branches $B_i$}
\label{fig. h geq 4 then not WTD 1}
\end{figure}

For each $i = 1,\dots, k$, consider the union of radars 
$$
R_i := R_{G_i}(w_i,1) \cup \left(\bigcup_{j = 1}^{\infty} R_{G_i}(w_i, 3 + 2j)\right)
$$ 
and 
$$
R_0 := R_{G_0}(w_0,3) \cup \left(\bigcup_{j = 1}^{\infty} R_{G_0}(w_0,5 + 2j)\right).
$$ 

Next, notice that each $R_i$ with $i > 0$ is an RDS of $G_i$, and $R_0$ dominates all red vertices in $G_0$ but $w_0$.
Since $R_i$ is an RDS in $G_i$ for $i > 0$, there exists a minimal RDS $D_i \subseteq R_i$ for each $i > 0$.
Also, there is a minimal set $D_0 \subseteq R_0$ which dominates all red vertices in $G_0$ but $w_0$.
Set $D_w := \cup_{i = 0}^k D_i$.

Next, let $G_x$ and $G_s$ be the subgraphs of $G$ induced by the sets $\br_v(x) \setminus \br_v(s)$ and $\br_v(s) \setminus \{\ell\}$, respectively (see Figure~\ref{fig. h geq 4 then not WTD 2}) and set
$$
R_x := R_{G_x}(x,2) \cup \left(\bigcup_{i = 0}^{\infty}R_{G_x}(v,6+2i)\right) \ \ \ \text{ and }\ \ \ R_s := \bigcup_{i = 0}^{\infty} R_{G_s}(y,3 + 2i).
$$ 

\begin{figure}[ht]
\centering
\includegraphics[scale = 1.25]{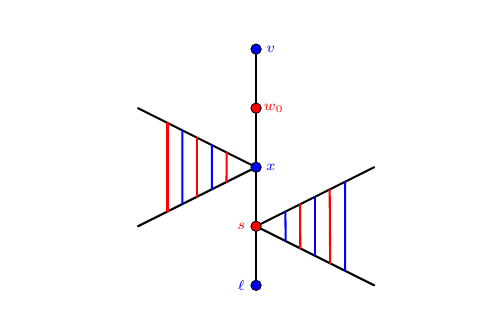}
\caption{Branches $B_x$ and $B_s$ in $T$}
\label{fig. h geq 4 then not WTD 2}
\end{figure}

Then $R_x$ dominates all red vertices in $G_x$, and $R_s$ dominates all red vertices in $G_s$ but $s$.
Let $D_x \subseteq R_x$ and $D_s \subseteq R_s$ be minimal sets of $R_x$ and $R_s$.
We claim that the sets 
$$
D := D_w \cup D_x \cup D_y \cup \{v,\ell\}
$$ 
and 
$$
D' := D_w \cup D_x \cup D_y \cup \{x\}
$$ 
are minimal RDSs in $T$.
Since $N_T( D_w \cup D_x \cup D_s) = V_\R(T) \setminus \{w_0,s\}$, $N_T(\{v,\ell\}) \supseteq \{w_0,s\}$, and $N_T(\{x\}) \supseteq \{w_0,s\}$, the sets $D$ and $D'$ are RDSs in $T$.

Now we show that $D$ and $D'$ are minimal.
First note that the open neighborhoods of $D_w$, $D_x$, and $D_y$ are all mutually disjoint.
Hence by Corollary~\ref{cor. neighbor disjoint min. sets  -> union is min.}, the set $Z := D_w \cup D_x \cup D_y$ is minimal in $T$.
Thus there exists a domination selector $\mathcal{D}_Z:Z \to N_T(Z)$.
Set $W := \set{w_0,\dots,w_k}$.
For each $i = 1, \dots,k$, we may choose $\mathcal{D}_Z(D_{w_i})$ so that $\mathcal{D}_Z(D_{w_i}) \cap W = \emptyset$ since $R_{G_{i}}(w_i, 3) \cap D_{w_i} = \emptyset$.
Also, as $w_0 \not\in N_T(D_{w_0})$, we have $\mathcal{D}_Z(D_w) \cap W = \emptyset$.
The set $\set{v,\ell}$ is minimal with domination selector $\mathcal{D}_{v\ell}:\set{v,\ell} \to N_T(\set{v,\ell})$ given by
$$
\mathcal{D}_{v\ell}(v') := \begin{cases}
    w_0 & v' = v\\
    s & v' = \ell\ .
\end{cases}
$$
Since we have $N_T(Z) \cap N_T(\set{v,\ell}) = W\setminus \set{w_0}$, the conditions $\mathcal{D}_Z(Z) \cap (W \setminus \set{w_0}) = \emptyset$ and $\mathcal{D}_{v\ell}(\set{v,\ell}) \cap (W\setminus \set{w_0}) = \emptyset$ show that $D$ is minimal by Lemma~\ref{lem. D(v) not in intersection -> union is minimal}.

To show that $D'$ is minimal, we choose $\mathcal{D}_Z$ such that for each $\alpha \in D_x \cap R_{G_x}(x,2)$, pick $\beta \in N_{G_x}(\alpha) \cap R_{G_x}(x,3)$ and set $\mathcal{D}_Z(\alpha) = \beta$ (which is possible since $D_x \cap R_{G_x}(x,4) = \emptyset$). 
This choice of $\mathcal{D}_Z$ gives us $\mathcal{D}_Z(D_x) \cap R_{G_x}(x,1) = \emptyset$.
Note that $\set{x}$ is minimal with domination selector $\mathcal{D}_x:\set{x} \to N_T(x)$ given by $\mathcal{D}_x(x) = w_0$.
Since $N_T(Z) \cap N_T(x) \subseteq R_{G_x}(x,1)$, $\mathcal{D}_Z(Z) \cap R_{G_x}(x,1) = \emptyset$, and $\mathcal{D}_x(\set{x}) \cap R_{G_x}(x,1) = \emptyset$, $D'$ is minimal by Lemma~\ref{lem. D(v) not in intersection -> union is minimal}.
Since $|D| > |D'|$, the tree $T$ has two minimal RDSs of difference sizes as desired.
This contradicts our initial supposition.
Finally, Corollary~\ref{cor. WTD iff RDBD WTD} implies that $T$ is not WTD.
\end{proof}

\begin{example}
\label{depth 4 implies not WTD example}
Consider the tree $T$ in Figure~\ref{depth 4 implies not WTD figure}.
\begin{figure}[ht]
\centering
\includegraphics[scale = 1]{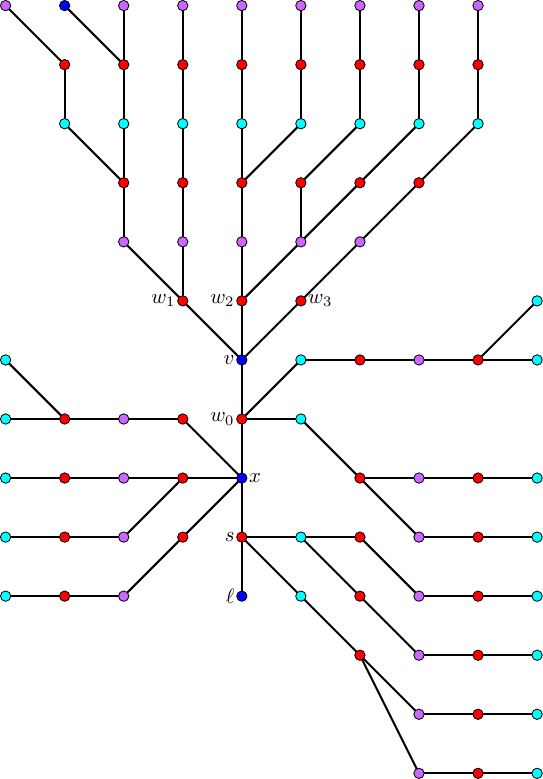}
\caption{a balanced tree $T$ of height greater than 3}
\label{depth 4 implies not WTD figure}
\end{figure}
We show that there are two minimal RDSs of different sizes in $T$.
We use the notation as in Theorem~\ref{thm. height >= 4 is not WTD}.
Let $Z$ be the set of all \textcolor{cczzff}{purple} vertices ($Z = D_w \cup D_x \cup D_s$).
The \textcolor{cyan}{cyan} vertices represent the blue vertices (even height vertices) excluded when we construct $R_x,R_s$, and $R_i$ for $i= 1,2,3$ in this example. 
Then the sets $D := Z \cup \set{v,\ell}$ and $D' := Z \cup \set{x}$ form minimal RDSs.
\end{example}

With Theorem~\ref{thm. height >= 4 is not WTD}, our next result shows that an WTD balanced tree must have height 0,1, or 3.

\begin{theorem}
\label{thm. height(T)=2 -> not WTD}
Let $T$ be a balanced tree with $\hei(T) = 2$.
Then $T$ has two leaves of distance 4 apart, hence $T$ is not WTD.
\end{theorem}

\begin{proof}
Since $\hei(T) = 2$, there exists a vertex $v \in V_2$. 
Then there is a leaf $l \in V_0$ and a support vertex $s \in V_1$ such that $\pth(v,l) = (v,s,l)$.
Since $\hei(v) = 2 = \hh(T)$, the vertex $v$ is only adjacent to height 1 vertices.
Also, the degree of $v$ must be at least 2 (else $v$ is a leaf).
So, there is a vertex $s' \neq s$ of height $1$ where $vs' \in E(V)$.
Since $\hh(s') = 1$, there is a leaf $l' \neq l$ which is adjacent to $s'$.
Since $\pth(l,l') = (l,s,v,s',l')$ has length 4, we have $\dist(l,l') = 4$. 
Therefore, by Theorem~\ref{thm. P5 -> Mix}, the tree $T$ is not WTD.
\end{proof}

\begin{theorem}
\label{thm. |N(s) cap V_2| > 1 -> not WTD}
Let $T$ be a balanced tree with blue leaves.
If there is a support vertex $s \in V_1$ such that $|N(s) \cap V_2| \geq 2$, then $T$ has two minimal RDSs of different sizes, hence $T$ is not WTD.
\end{theorem}

\begin{proof}
    If $\hh(T) \geq 4$, then $T$ has two minimal RDSs of different sizes by Theorem~\ref{thm. height >= 4 is not WTD}.
    So we assume that $\hh(T) \leq 3$.
    Since $V_2 \neq \emptyset$ by assumption, $\hh(T) \geq 2$.
    If $\hh(T) = 2$, then $T$ has two minimal RDSs of different sizes and is not WTD by Corollary~\ref{cor. WTD iff RDBD WTD} and Thoerem~\ref{thm. height(T)=2 -> not WTD}.
    Hence we assume that $\hh(T) = 3$.
    
    If there is a vertex $x \in V_2$ with $|N(x) \cap V_1| \geq 2$, then $T$ has two minimal RDSs of different sizes and is not WTD by Corollary~\ref{cor. WTD iff RDBD WTD}, Theorem~\ref{thm. P5 -> Mix}, and Lemma~\ref{lem. no 4-path iff |N(v) cap V_1| = 1}.
    Thus we also assume that for all $x \in V_2$, we have $|N(x) \cap V_1| = 1$.

    Let $l$ be a leaf adjacent to $s$.
    Since $|N(s) \cap V_2| \geq 2$, there are two distinct vertices $v_1,v_2$ in $N(s) \cap V_2$.
    Since $\hh(v_1) = \hh(v_2) = 2$, $v_1$ and $v_2$ are not leaves.
    Hence $\deg(v_1)$ and $\deg(v_2)$ are at least 2.
    By assumption, we have $|N(v_i) \cap V_1| = 1$ for $i = 1,2$.
    Thus $v_i$ must be adjacent to a height 3 vertex $u_i$ for $i= 1,2$.
    Similarly, since $u_1$ and $u_2$ are not leaves, there are distinct vertices $v_1',v_2' \in V_2$, distinct support vertices $s_1,s_2 \in V_1$, and distinct leaves $l_1,l_2 \in V_0$ such that $\pth(u_i,l_i) = (u_i,v_i',s_i,l_i)$ for $i = 1,2$.
    Let $T'$ be the subgraph of $T$ induced by the set $\set{l,l_1,l_2,s,s_1,s_2,v_1,v_1',v_2,v_2',u_1,u_2}$  (Figure~\ref{Figure3}).
    \begin{figure}[ht]
    \centering
    \includegraphics{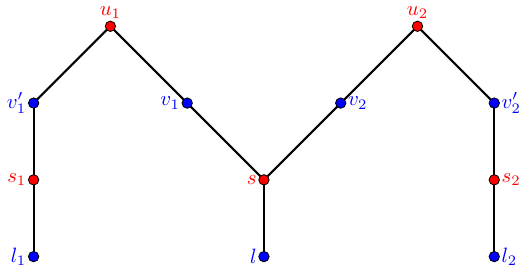}
    \caption{Induced subgraph $T'$ of $T$}
    \label{Figure3}
    \end{figure}
    
    Next, we construct a set dominating every red (odd height) vertex that is not in $T'$.
    Set
    \begin{align*}
        &X_3 := (N_T(V_2(T')) \cap V_3(T)) \setminus V_3(T') & & X_2 := (N_T(X_3) \cap V_2(T)) \setminus V_2(T')\\
        &X_1 := (N_T(X_2) \cap V_1(T)) \setminus V_1(T') & & X_0 := (N_T(X_1) \cap V_0(T))\setminus V_0(T')\\ 
        &Y_2 := (N_T(X_1) \cap V_2(T)) \setminus (X_2 \cup V_2(T')) &&
    \end{align*}
 see Figure~\ref{fig. support adj to V2 1}. 
 We claim that the set 
    $$
    U := (V_0(T) \cup V_2(T)) \setminus (X_0 \cup Y_2 \cup N_T(V_1(T') \cup V_3(T')))
    $$ 
    has $N_T(U) = (V_1(T) \cup V_3(T)) \setminus (V_1(T') \cup V_3(T'))$.
    \begin{figure}[ht]
    \centering
    \includegraphics[scale = 0.6]{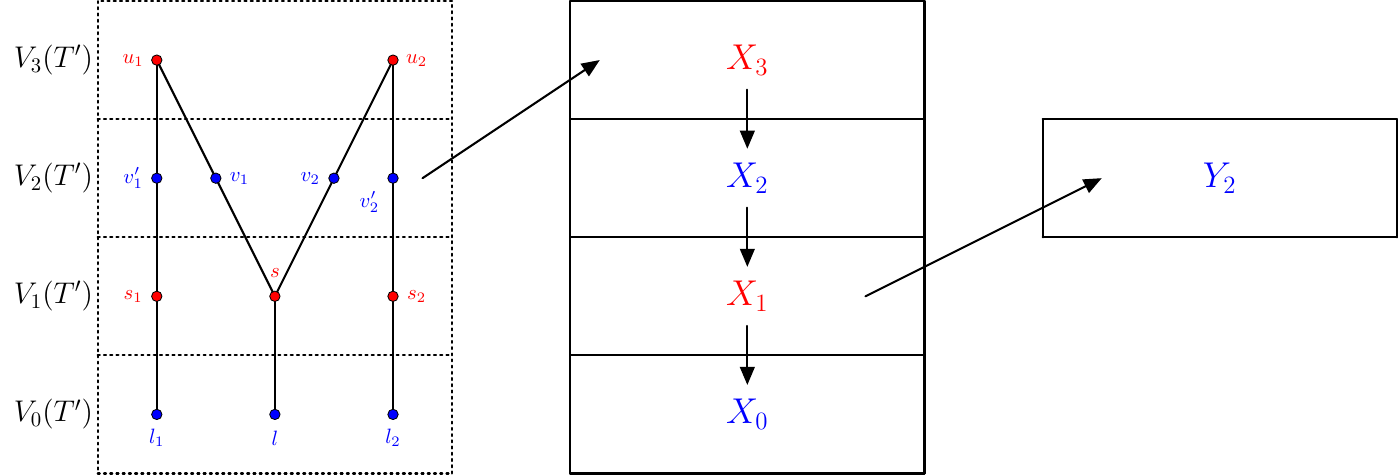}
    \caption{Vertex sets $X_3$, $X_2$, $X_1$, $X_0$, and $Y_2$}
    \label{fig. support adj to V2 1}
    (Arrows indicate the adjacency relation between the sets; for instance, $Y_2 \subseteq N_T(X_1)$)
    \end{figure}
    
    The containment $N_T(U) \subseteq (V_1(T) \cup V_3(T)) \setminus (V_1(T') \cup V_3(T'))$ is straightforward since $U \subseteq V_0(T) \cup V_2(T)$ implies $N_T(U) \subseteq V_1(T) \cup V_3(T)$ because $T$ is a balanced tree, and $N_T(V_1(T') \cup V_3(T')) \cap U = \emptyset$ implies $N_T(U) \cap (V_1(T') \cup V_3(T')) = \emptyset$.

    Before we prove the reverse containment, we show that $X_2 \subseteq U$.
    Since $X_2 \subseteq V_2(T)$, we need to show that $X_2 \cap (X_0 \cup Y_2 \cup N_T(V_1(T') \cup V_3(T'))) = \emptyset$.
    By the constructions of $X_0$ and $Y_2$, we get $X_2 \cap (X_0 \cup Y_2) = \emptyset$.
    To show that $X_2 \cap N_T(V_1(T') \cup V_3(T')) = \emptyset$, assume by way of contradiction that there is a vertex $c \in X_2 \cap N_T(V_1(T') \cup V_3(T'))$.
    We construct a cycle containing the vertex $c$, contradicting that $T$ is a tree.
    First, suppose that $c \in N_T(V_1(T'))$.
    Then there is a vertex $s' \in V_1(T')$ that is adjacent to $c$.
    By construction of $X_2$ and $X_3$, there are vertices $x_3 \in X_3$ and $v'$ in $V_2(T')$ such that $cx_3 \in E(T)$ and $v'x_3 \in E(T)$.
    Since $T'$ is a tree and $s'$ and $v'$ are vertices in $T'$, there is a path $P$ in $T'$ from $v'$ to $s'$.
    Thus the paths $P$ and $(s',c,x_3,v')$ form a cycle in $T$.
    Similarly, we can construct a cycle when $c \in N_T(V_3(T'))$.

    Now we show that $N_T(U) \supseteq (V_1(T) \cup V_3(T)) \setminus (V_1(T') \cup V_3(T'))$.
    Let $x \in (V_1(T) \cup V_3(T))$ such that $x \not\in (V_1(T') \cup V_3(T'))$.
    We show that there is a vertex in $U$ that is adjacent to $x$.
    We have three cases to consider.

    \emph{Case 1.} $x \in X_3 \cup X_1$.
    If $x \in X_1$, then there is a vertex $x_2 \in X_2 \subseteq U$ that is adjacent to $x$.
    If $x \in X_3$, then $\deg(x) \geq 2$ as $x$ is not a leaf.
    Hence $x$ is adjacent to at least 2 vertices in $V_2(T)$.
    By construction of $X_3$, $x$ is adjacent to a vertex in $V_2(T')$.
    Notice that $x$ cannot be adjacent to 2 vertices in $V_2(T')$ since this will form a cycle in $T$.
    Thus $x$ must be adjacent to a vertex $y \in V_2(T)$ such that $y \not\in V_2(T')$.
    By construction of $X_2$, $y \in X_2 \subseteq U$ as desired.

    \emph{Case 2.} $x \in V_1(T)\setminus X_1$. Then there is a leaf $l_x \in V_1(T)$ that is adjacent to $x$.
    Since $x \not\in X_1$, $l_x \not\in X_0$ as $x$ is the only vertex $l_x$ is adjacent to.
    We also have $l_x \not\in Y_2 \cup N_T(V_3(T'))$ since vertices in $Y_2 \cup N_T(V_3(T'))$ have height 2.
    Finally, $l_x \not\in N_T(V_1(T'))$ since we assumed that $x \not\in V_1(T')$.
    Thus we have $l_x \in U$ as desired.
    
    \emph{Case 3.} $x \in V_3(T)\setminus X_3$.
    Then $x$ is not a leaf, hence $\deg_T(x) \geq 2$. 
    Also, as $\hh(T) = 3=\hh(x) = 3$ and $T$ is balanced, $x$ is adjacent to at least two height 2 vertices $t_1,t_2 \in V_2(T)$.
    To show that at least one of $t_1$ or $t_2$ is in $U$, assume by way of contradiction that $t_1,t_2 \not\in U$.
    Then $t_1,t_2 \in X_0 \cup Y_2 \cup N_T(V_1(T') \cup V_3(T'))$ by definition of $U$.
    Since height of $t_1$ and $t_2$ is 2, we have $t_1,t_2 \in Y_2 \cup N_T(V_1(T') \cup V_3(T'))$.
    As in the proof of $X_2 \subseteq U$ there is a path from $t_1$ to $t_2$ only using the vertices in $V(T') \cup \left(\bigcup_{i = 1}^3X_i\right) \cup Y_2 \cup N_T(V_1(T') \cup V_3(T'))$. 
    But $(t_1, x, t_2)$ is also a path from $t_1$ to $t_2$, with $x$ not belonging to $V(T') \cup \left(\bigcup_{i = 1}^3X_i\right) \cup Y_2 \cup N_T(V_1(T') \cup V_3(T'))$.
    Thus we have a cycle in $T$, a contradiction.

    Therefore, we have $N_T(U) = (V_1(T) \cup V_3(T)) \setminus (V_1(T') \cup V_3(T'))$ as desired.

    Now let $D \subseteq U$ be a minimal set of $U$.
    Let $D_1 := \set{l_1,l_2,v_1,v_2}$ and $D_2 := \set{l,v_1',v_2'}$.
    We claim that the sets $D \cup D_1$ and $D \cup D_2$ are minimal RDSs.
    For $i = 1,2$, since $T$ is balanced and $N_T(D_i) \supseteq \set{u_1,u_2,s,s_1,s_2} = V_1(T') \cup V_3(T')$, we get $N_T(D \cup D_i) = V_1(T) \cup V_3(T)$.

    To show the minimality of $D \cup D_1$ and $D \cup D_2$, define domination selectors for $D_1$ and $D_2$
    $$
    \calD_1(v) := \begin{cases}
    s_1 & v = l_1\\
    s_2 & v = l_2\\
    u_1 & v = v_1\\
    u_2 & v = v_2
    \end{cases}
    \qquad
    \calD_2(v) := \begin{cases}
        s_1 & v = v_1'\\
        s_2 & v = v_2'\\
        s & v = l\ \ \ .
    \end{cases}
    $$

    For $D$, we claim that there is a domination selector $\calD: D \to N_T(D)$ such that for all $x \in D \cap X_2$, we have $\calD(x) \in X_1$.
    To this end, we need to show that for all $x \in D \cap X_2$, there is a vertex $y \in X_1$ such that $N_T(y) \cap D = \set{x}$.
    So, assume by way of contradiction that there is a vertex $x \in D \cap X_2$ such that for all $y\in X_1$, we have $N_T(y) \cap D \neq \set{x}$.
    Since $\hh(x) = 2$, there is a support vertex $s' \in V_1(T)$ that is adjacent to $x$.
    By construction of $X_2$, either $s' \in X_1$ or $s' \in V_1(T')$.
    To show that $s' \in X_1$, assume by way of contradiction that $s' \in V_1(T)'$.
    Then one can show that this assumption creates a cycle containing $s'$ and $x$ using the vertices in $V(T') \cup X_2 \cup X_3$, a contradiction.
    Thus $s' \in X_1$.
    Since $N_T(s') \cap D \neq \set{x}$ but $x \in N_T(s') \cap D$, there is a vertex $x' \in D \cap X_2$ such that $N_T(s') \cap D \supseteq \set{x,x'}$.
    By the construction of $X_2$ and $X_3$, there is a path $P$ from $x$ to $x'$ only using the vertices in $V(T') \cup X_3$.
    On the other hand, $(x,s',x')$ is a path from $x$ to $x'$ which is not $P$ as $s' \not \in V(T') \cup X_3$.
    Hence we have a cycle in $T$, a contradiction, and the claim is established.

    Furthermore, since the only vertices in $U$ that are adjacent to $X_3$ are in $X_2$, we get $\calD(D) \cap X_3 = \emptyset$ as $\calD(x) \in X_1$ for all $x \in D \cap X_2$.

    Now consider $N_T(D) \cap D_i$ for $i = 1,2$.
    By the above computation of $N_T(U) = N_T(D)$, we have $N_T(D) \cap D_i =\emptyset$ for $i = 1,2$.
    Since $\calD_i(D_i) \cap X_3 = \emptyset$ for $i = 1,2$, and $\calD(D) \cap X_3 = \emptyset$, the sets $D \cup D_1$ and $D \cup D_2$ are minimal by Lemma~\ref{lem. D(v) not in intersection -> union is minimal}. 
\end{proof}  

The following result is a converse to some previous results when $\hh(T) = 3$.

\begin{theorem}
\label{thm. h(T) = 3 and no P4 and |N(s) cup V2| = 2 -> unmimxed}
Let $T$ be a balanced tree of height 3.
Suppose that  
\begin{enumerate}
    \item for all $u \in V_2$, $|N(u) \cap V_1| = 1$, and
    \item for all $s \in V_1$,  $|N(s) \cap V_2| \leq 1$.
\end{enumerate}
Then $T$ is WTD.
\end{theorem}

\begin{proof}
First, we show some structural properties of $T$.
By Lemma~\ref{lem. LeafAdding}, the WTD-ness of $T$ is independent of the numbers of leaves attached to the support vertices.
Hence we assume that every support vertex is adjacent to exactly one leaf, which implies that $|V_0| = |V_1|$.
So, for each $s \in V_1$, denote the unique leaf adjacent to $s$ by $\ell(s)$, and for each leaf $\ell \in V_0$, denote the unique support vertex adjacent to $\ell$ by $s(\ell)$.

Next, since $T$ is connected and balanced with $\hh(T) = 3$, every support vertex must be adjacent to at least one height-2 vertex because there is a path between every support vertex and height-3 vertex.
Hence the inequality in (2) becomes an equality: for all $s \in V_1$, $|N(s) \cap V_2| = 1$.

Now we claim that $|V_1| = |V_2|$.
To show the claim, assume by way of contradiction that $|V_2| \neq |V_1|$.
First, suppose that $|V_2| > |V_1|$.
By (1) and the pigeonhole principle, there exists a support vertex $c \in V_1$ that is adjacent to at least two height-2 vertices.
This contradicts~(2).
Now suppose that $|V_2| < |V_1|$.
Then by (2) and the pigeonhole principle, there exists a vertex $c \in V_2$ that is adjacent to at least two support vertices, contradicting (1).
Therefore, we get $|V_2| = |V_1|$.

The hypotheses (1)--(2) and the condition $|V_2| = |V_1|$ implies that the set of edges 
$$
E(V_1,V_2) := \Set{uv \in E}{u \in V_1, v \in V_2}
$$
forms a bijection (perfect matching) between $V_1$ and $V_2$.
Hence for each $v \in V_2$, let $s(v)$ be the unique support vertex adjacent to $v$; i.e., $s(v)v \in E(V_1,V_2)$.

To show that $T$ is WTD, we use Corollary~\ref{cor. WTD iff RDBD WTD}.
First, we show that every minimal BDS of $T$ has the same size.
More specifically, we show that $V_1$ is the only minimal BDS of $T$.
Let $D$ be a minimal BDS of $T$.
Then $N(D) = V_0 \cup V_2$.
Since $V_0 \subset N(D)$ and the support vertices are the only vertices that are adjacent to the leaves, we have $V_1 \subseteq D$.
But, we also have $N(V_1) = V_0 \cap V_2$.
By mininmality of $D$, we have $V_1 = D$.
Thus $V_1$ is a minimal BDS contained in any BDS, showing that $V_1$ is the only minimal BDS in $T$.

Next, we show that every minimal RDS has the same size.
First, notice that the set $V_2$ is a minimal RDS because it dominates $V_1 \cup V_3$ and each $s \in V_1$ is dominated by a unique $x \in V_2$.
Let $D$ be a minimal RDS.
We show that $|D| = |V_2|$.
Since there are no vertices $c \in D \subseteq V_0 \cup V_2$ with $|N(c) \cap V_1| \geq 2$, we have $|D| \geq |V_1|$. 
Next, we show that $|D| \leq |V_1|$.
Let $s \in V_1$.
Then there are vertices $l \in V_0$ and $u \in V_2$ such that $N(s) = \set{l,u}$ by the assumptions on $T$.
Since $N(l) = \set{s} \subseteq N(u)$, the minimality of $D$ implies that $|D \cap N(s)| \leq 1$.
Since $D \subseteq V_0 \cup V_2 = \bigcup_{t \in V_1}N(t)$ where the union is disjoint, we get
$$
|D| = \left|\bigcup_{t \in V_1} (D \cap N(t))\right| = \sum_{ t \in V_1} |D \cap N(t)| \leq \sum_{ t \in V_1} 1 = |V_1|\ .
$$
This gives $|D| = |V_1|$ hence every minimal RDS of $T$ has the same size.
\end{proof}

Now we present our descriptive characterization of WTD balanced trees.

\begin{theorem}
\label{thm. char. of WTD delt. trees}
Let $T$ be a balanced tree.
Then $T$ is WTD if and only if
\begin{enumerate}
    \item $\hei(T) \leq 3$,
    \item for all $v \in V_2$, $|N(v) \cap V_1| = 1$, and
    \item for all $v \in V_1$, $|N(v) \cap V_2| \leq 1$.
\end{enumerate}
\end{theorem}

\begin{proof}
($\Rightarrow$) Suppose that $T$ is WTD.
Then Theorem~\ref{thm. height >= 4 is not WTD} gives (1), Theorem~\ref{thm. P5 -> Mix} and Lemma~\ref{lem. no 4-path iff |N(v) cap V_1| = 1} give (2), and Theorem~\ref{thm. |N(s) cap V_2| > 1 -> not WTD} and  gives (3).

($\Leftarrow$) Suppose that (1), (2), and (3) hold.
By (2), we have $\hh(T) \neq 2$ as height 2 vertices cannot be leaves.
Hence by (1), we have $\hh(T) = 0,1$, or $3$.
If $\hh(T) = 0$, then $T$ is a graph with a single vertex, which is vacuously WTD. 
If $\hh(T) = 1$, then it is straightforward to show that $T$ is a graph with a single support vertex $s$ with at least 2 leaves attached (such graphs are called star graphs, see Figure~\ref{fig. star graph no label}). 
In this case, every minimal TDS of $T$ is of the form $\set{s,l}$ where $l$ is a leaf, hence $T$ is WTD.
\begin{figure}[ht]
    \centering
    \includegraphics[scale = 1]{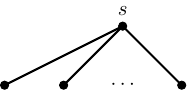}
    \caption{Balanced tree of height 1 (star graph)}
    \label{fig. star graph no label}
\end{figure}
If $\hei(T) = 3$, then by Theorem~\ref{thm. h(T) = 3 and no P4 and |N(s) cup V2| = 2 -> unmimxed}, $T$ is WTD.
\end{proof}

Theorem~\ref{thm. char. of WTD delt. trees} and the proof of Theorem~\ref{thm. h(T) = 3 and no P4 and |N(s) cup V2| = 2 -> unmimxed} show that if $T$ is an WTD balanced tree with $\hh(T) = 3$, then the inequality in (3) becomes an equality.

\section{Descriptive characterization of arbitrary trees via interior graphs}\label{sec. interior graphs of trees}

In this section, we define the interior graphs of a tree $T$ and prove that $T$ is WTD if and only if its interior graphs are WTD; see Corollary~\ref{cor. T WTD <-> TBTR WTD}.
We combine this result with Theorem~\ref{thm. char. of WTD delt. trees} to obtain our descriptive characterization of WTD trees in Theorem~\ref{thm. char. of WTD trees}.

\begin{assumptions}
    In this section, we let $T$ be a tree with a 2-coloring $\chi:V(T) \to \set{\R,\B}$.
\end{assumptions}

\begin{definition}\label{def. interior graphs}
    We define the \emph{blue interior graph} of $T$ to be the subgraph $T_\B$ of $T$ induced by the set 
    $$
    V \setminus N[V_1 \cap \chi^{-1}(\B)]
    $$
    where $N[U] := N(U) \cup U$ is the \emph{closed neighborhood of $U$} for any $U \subseteq V$; i.e., $T_B$ is the subgraph of $T$ induced by deleting all blue support vertices together with their neighbors.
    Similarly, we define $T_\R$ to be the \emph{red interior graph} of $T$ induced by the set
    $$
    V \setminus N[V_1 \cap \chi^{-1}(\R)].
    $$
\end{definition}

\begin{example}\label{ex. decomp. of a tree}
Figure \ref{fig. 2-colored T} shows a 2-colored tree $T$, and Figure~\ref{fig. T_B T_R of T} below shows its blue and red interior graphs. This example shows that the interior graphs can be forests.
\begin{figure}[ht]
        \centering
        \includegraphics{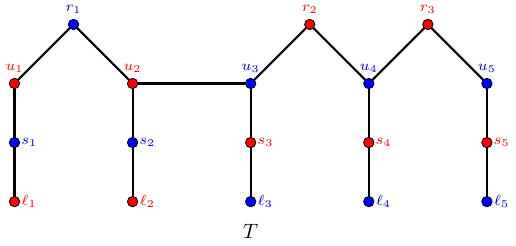}
        \caption{A 2-coloring of a tree $T$}
        \label{fig. 2-colored T}
\end{figure}
\begin{figure}[ht]
        \centering
        \includegraphics{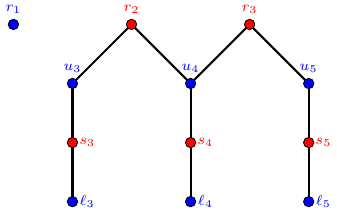}
        \qquad
        \vline
        \qquad
        \includegraphics{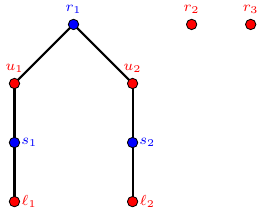}
        \caption{Interior graphs $T_\B$ (left) and $T_\R$ (right) of $T$ from Figure \ref{fig. 2-colored T}}
        \label{fig. T_B T_R of T}
\end{figure}    
\end{example}

\begin{example}\label{ex. T, T_B = T_R}
    Consider the tree $T$ in Figure \ref{fig. tree_with_all_h}.
    Its interior graphs $T_\B$ and $T_\R$ are shown in Figure  \ref{fig. T_B T_R of tree with all h}. 
    For this specific tree $T$, its interior graphs are isomorphic.

    \begin{figure}[ht]
        \centering
        \includegraphics{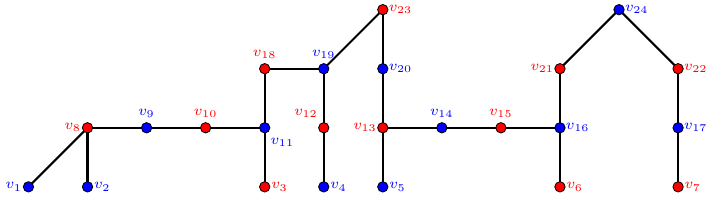}
        \caption{A tree $T$}
        \label{fig. tree_with_all_h}
    \end{figure}

    \begin{figure}[ht]
        \centering
        \includegraphics[width = 0.449\textwidth]{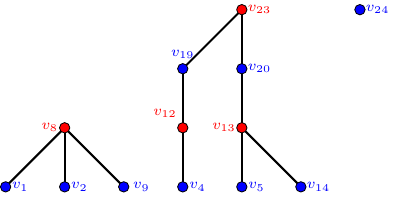}
        \quad
        \vline
        \quad
        \includegraphics[width = 0.449\textwidth]{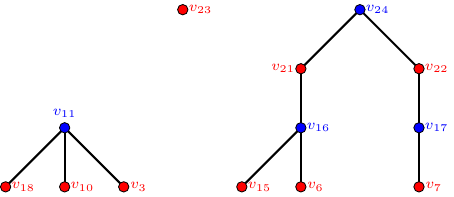}
        \caption{$T_\B$ (left) and $T_\R$ (right) of $T$ from Figure \ref{fig. tree_with_all_h}}
        \label{fig. T_B T_R of tree with all h}
    \end{figure}
\end{example}

Each connected component of the interior graphs in Examples  \ref{ex. T, T_B = T_R} and \ref{ex. decomp. of a tree} is a balanced tree.
This is true in general and is proved in the following lemma.

\begin{lemma}\label{lem. interior is delta forest}
    Every connected component of $T_\B$ and $T_\R$ is a balanced tree. 
    Hence we say that $T_\B$ and $T_\R$ are balanced forests.
\end{lemma}

\begin{proof}
    Without loss of generality, consider the blue interior graph $T_\B$.
    By Definition~\ref{def. interior graphs}, we delete all blue support vertices with their neighboring vertices which are red when we construct $T_\B$.
    Hence the leaves and the isolated vertices in $T_\B$ must be blue, implying that the connected components of $T_\B$ must be balanced trees by Proposition~\ref{prop. equiv. delt. tree}.
\end{proof}

The following results demonstrate the importance of studying the WTD-ness of the interior graphs of $T$ to understand the WTD-ness of $T$.

\begin{lemma}\label{lem. RDSs In TR <-> RDSs In T }
    Let $D \subseteq V(T)$. 
    Then $D$ is a minimal BDS in $T$ if and only if 
    $$
    D = (V_1(T) \cap V_\R(T)) \cup D'
    $$ 
    where $D'$ is a minimal BDS in $T_\R$.
    Similarly, $D$ is a minimal RDS of $T$ if and only if 
    $$
    D = (V_1(T) \cap V_\B(T)) \cup D'
    $$ 
    where $D'$ is a minimal RDS in $T_\B$.
\end{lemma}

\begin{proof}
It suffices to prove the case for BDSs.
Write $V_{1,\R} := (V_1(T) \cap V_\R(T))$.

For one implication, we assume that $D'$ is a minimal BDS in $T_\R$ and show that $D := V_{1,\R} \cup D'$ is a minimal BDS in $T$.
The set of blue vertices of $T$ that are not in $T_\R$ is $N_T(V_{1,\R})$ by Definition~\ref{def. interior graphs}, so $D$ is a BDS in $T$.
Let $\mathcal{D}_\R':D' \to N_{T_\R}(D')$ be a domination selector of $D'$ in $T_\R$, and let $\mathcal{D}'':V_{1,\R} \to N_T(V_{1,\R})$ be the function defined by
$$
\mathcal{D}''(s) = l_s
$$
for some choice of $l_s \in V_0(T) \cap N_T(s)$.
Since $T_\R$ is an induced subgraph, one can show that the function $\mathcal{D'}: D' \to N_T(D')$ given by
$$
\mathcal{D'}(x) = \mathcal{D}'_\R(x)
$$
is also a domination selector of $D'$ in $T$.
Next, note that we have
$$
N_T(D') \cap N_T(V_{1,\R}) \subseteq N_T(V_{1,\R})\setminus V_0(T).
$$
Since $\mathcal{D}'(D') = \mathcal{D}'_\R(D') \subseteq V(T_\R)$, we get $\mathcal{D}'(D') \cap (N_T(V_{1,\R})\setminus V_0(T)) = \emptyset$ by construction of $T_\R$.
On the other hand, since $\mathcal{D}''(V_{1,\R}) \subseteq V_0(T)$, we get $\mathcal{D}''(V_{1,\R}) \cap (N_T(V_{1,\R})\setminus V_0(T)) = \emptyset$.
Therefore, $D$ is minimal by Lemma~\ref{lem. D(v) not in intersection -> union is minimal}.

For the converse, suppose that $D$ is a minimal BDS of $T$.
Then $V_{1,\R} \subseteq D$ since all blue leaves in $D$ must be dominated.
Set $D' := D \setminus V_{1,\R}$.
We show that $D'$ is a minimal BDS of $T_\R$.
Indeed we have $N_{T_\R}(D') = V_\B(T_\R)$ since $N_T(V_{1,\R}) \cap V(T_\R) = \emptyset$ by Definition~\ref{def. interior graphs}, and $N_T(D) = V_\B(T)$.
If $D'$ is not minimal in $T_\R$, then there exists a minimal BDS $D'' \subsetneq D'$ in $T_\R$.
But by the previous implication, $V_{1,\R} \cup D'' \subsetneq D$ becomes a minimal BDS in $T$, contradicting the minimality of $D$.
Thus $D'$ must be a minimal BDS in $T_{\R}$.
\end{proof}

Here is a corollary of Lemma~\ref{lem. RDSs In TR <-> RDSs In T }.

\begin{corollary}\label{cor. T WTD <-> TR BD-WTD and TB RD-WTD}
Every minimal RDS in $T$ has the same size if and only if every minimal RDS in $T_\B$ has the same size.
Similarly, every minimal BDS in $T$ has the same size if and only if every minimal BDS in $T_\R$ has the same size.    
\end{corollary}

\begin{lemma}\label{lem. T WTD <-> T RD-WTD}
    Let $T$ be a balanced tree with blue leaves.
    Then $T$ is WTD if and only if every minimal RDS in $T$ has the same size. 
\end{lemma}

\begin{proof}
    The forward implication is by Corollary~\ref{cor. WTD iff RDBD WTD}.
    So, suppose that every minimal RDS of $T$ has the same size.
    Again by Corollary~\ref{cor. WTD iff RDBD WTD}, it suffices to show that every BDS of $T$ has the same size; in particular, we show that $T$ has a unique minimal BDS.
    Let $D \subseteq V$ be a minimal BDS.
    Since the leaves of $T$ are blue, $V_1 \subseteq D$.
    Also, as every RDS of $T$ has the same size, we get $\hh(T) \leq 3$ by Theorem~\ref{thm. height >= 4 is not WTD}. 
    Thus we have $V_\B = V_0 \cup V_2$.
    Thus $N(V_1) = V_\B$.
    By the minimality of $D$, we get $D = V_1$.
\end{proof}

Corollaries~\ref{cor. WTD iff RDBD WTD} and \ref{cor. T WTD <-> TR BD-WTD and TB RD-WTD} combine with Lemma~\ref{lem. T WTD <-> T RD-WTD} to establish our next result.

\begin{corollary}\label{cor. T WTD <-> TBTR WTD}
    Let $T$ be a tree with a 2-coloring.
    Then $T$ is WTD if and only if its interior graphs $T_\B$ and $T_\R$ are WTD.
\end{corollary}

We conclude this section with our descriptive characterization of WTD trees.
It follows directly from Corollary~\ref{cor. T WTD <-> TBTR WTD} and Theorem~\ref{thm. char. of WTD delt. trees}.

\begin{theorem}\label{thm. char. of WTD trees}
    A tree $T$ is WTD if and only if every connected component $T'$ of $T_\B$ and $T_\R$ has
    \begin{enumerate}
        \item $\hh(T') \leq 3$,
        \item for all $v \in V_2(T')$, we have $|N_{T'}(v) \cap V_1(T')| = 1$, and
        \item for all $v \in V_1(T')$, we have $|N_{T'}(v) \cap V_2(T')| \leq 1$.
    \end{enumerate}
\end{theorem}

\section{Constructing WTD balanced trees}\label{sec. const. WTD bal. trees}

The main result of this section, Theorem~\ref{thm. constructing hh = 3 delt. trees}, gives a constructive characterization of WTD balanced trees.

\begin{assumptions}
    Set $\mathbb{N}_0 := \mathbb{N} \cup \set{0}$. 
    For $n \in \mathbb{N}_0$, define $[n] := \set{1,2,\dots,n}$ and $[n]_0 := \set{0,1,\dots,n}$.
    For every $n \in \mathbb{N}_0$, we denote $P_n$ to be the path graph with $V(P_n) = [n]_0$ and $E(P_n) = \set{(0,1),(1,2),\dots,(n-1,n)}$.
\end{assumptions}

We apply the following operation to balanced trees below.

\begin{definition}
    For two graphs $G_1 = (V',E')$ and $G_2 = (V'',E'')$ with vertices $v' \in V'$ and $v'' \in V''$, we set $e:= v'v''$ and define the \emph{edge join} of $G_1$ and $G_2$ by $e$ to be the graph 
    $$
    G_1 +_e G_2 := (V'\cup V'', E' \cup E'' \cup \set{v'v''}).
    $$
    We define an operation $\mathcal{O}$ as follows:
    Let $T = (V,E)$ be a tree, and let $v \in \bigcup_{i = 1}^3V_i$.
    
    \begin{enumerate}
        \item If $\hh(v) = 1$, then 
        $$
        \mathcal{O}(T,v) := T +_{(v,0)} P_0.
        $$  
        \item  If $\hh(v) = 2$, then
        $$
        \mathcal{O}(T,v) := T +_{(v,3)} P_3.
        $$
        \item  If $\hh(v) = 3$, then
        $$
        \mathcal{O}(T,v) := T +_{(v,2)} P_2.
        $$
    \end{enumerate}
    In other words, the operator $\mathcal{O}$ adds a ``whisker" of length 1, 4, or 3 to $T$, depending on $\hh(v)$. 
    In particular $\mathcal{O}(T,v)$ is a tree, so for $v_1 \in \bigcup_{i = 1}^3V_i$ after appropriate relabeling so that $V(\mathcal{O}(T,v_1)) \cap \mathbb{N}_0 = \emptyset$, let $v_2 \in \bigcup_{i = 1}^3V_i(\mathcal{O}(T,v_1))$, and set
    $$
    \mathcal{O}(T,(v_1,v_2)) := \mathcal{O}(\mathcal{O}(T,v_1),v_2).
    $$
    Inductively, given $n \in \mathbb{N}$ and $v_j \in \bigcup_{i = 1}^3V_i\left(\mathcal{O}(T,(v_1,\dots,v_{j-1}))\right)$ for $j \in [n]$, we set
    $$
    \mathcal{O}(T,(v_1,\dots,v_n)) := \mathcal{O}(\mathcal{O}(T,v_1,\dots,v_{n-1}),v_n).
    $$
\end{definition}

\begin{example}
    Consider the following graph $T$ in Figure~\ref{fig: P_6} that is isomorphic to $P_6$.
    \begin{figure}[ht]
        \centering
        \includegraphics{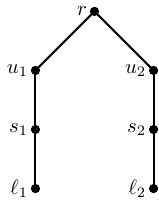}
        \caption{Graph $T \cong P_6$}
        \label{fig: P_6}
    \end{figure}
The graphs $\mathcal{O}(T,u_2)$, and $\mathcal{O}(T,r)$ are given in Figure~\ref{fig: P_6 O_2 O_3}.
\begin{figure}[ht]
    \centering
    \includegraphics{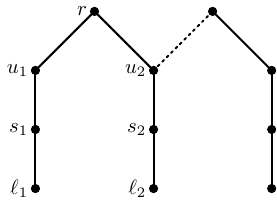}
    \qquad
    \qquad
    \includegraphics{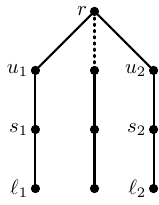}
    \caption{Graphs $\mathcal{O}(T,u_2)$ and $\mathcal{O}(T,r)$.}
    \label{fig: P_6 O_2 O_3}
\end{figure}
\end{example}

\begin{fact}\label{fact. O123 preserve WTD-ness}
    Let $T$ be a balanced tree of height 3.
    Let $v \in V\setminus V_0$.
    Then $T$ is WTD if and only if $\mathcal{O}(T,v)$ is WTD, by Theorem~\ref{thm. char. of WTD delt. trees}, since $\mathcal{O}$ preserves the height, and conditions (2) and (3) of Theorem~\ref{thm. char. of WTD delt. trees}
\end{fact}

The proof of Theorem~\ref{thm. char. of WTD delt. trees} shows that $P_0$ is the unique height-0 WTD balanced tree, and every height-1 WTD balanced tree is of the form $\mathcal{O}(P_2,(s,\dots,s))$ where $V_1(P_2) = \set{s}$.
Since Theorems~\ref{thm. height >= 4 is not WTD} and \ref{thm. height(T)=2 -> not WTD} show that an WTD balanced tree must have height 0,1, or 3, it remains to construct every height-3 WTD balanced tree.
We begin this process with the following lemma; the graph in Figure~\ref{Figure3} shows why $T$ needs to be WTD and balanced for the conclusion to hold.

\begin{lemma}\label{lem. subgraph V32 is a tree}
    Let $T = (V,E)$ be an WTD balanced tree of height 3.
    Let $T' = (V',E')$ be the subgraph of $T$ induced by the set $V_3 \cup V_2$.
    Then $T'$ is a tree, i.e., $T'$ is connected.
\end{lemma}

\begin{proof}
    Recall the following property of trees: Let $G$ be a graph with a leaf $\ell$. Then $G$ is a tree if and only if $G\setminus \ell$ is a tree.
    Next, consider the subgraph $T''$ of $T$ induced by the set $V_1 \cup V_2 \cup V_3$.
    Then $T''$ is obtained by only deleting the leaves of $T$, hence $T''$ is a tree.
    By Theorem~\ref{thm. char. of WTD delt. trees}, the vertices in $V_1(T)$ must be leaves in $T''$.
    Since $T'$ can also be obtained by deleting the leaves of $T''$ which are the vertices in $V_1(T)$, $T'$ is also a tree.
\end{proof}

The following corollary is critical for the proof of Theorem~\ref{thm. constructing hh = 3 delt. trees}.

\begin{corollary}\label{cor. hh = 2 of upper deg 1 exists}
    Let $T$ be an WTD balanced tree of height 3.
    Then there exists $u \in V_2$ such that $|N(u) \cap V_3| = 1$.
\end{corollary}

\begin{proof}
    First, we show that for all $u \in V_2$, we have $|N(u) \cap V_3| \geq 1$.
    Assume by way of contradiction that there is a vertex $c \in V_2$ with $|N(c) \cap V_3| = 0$.
    Since $T$ balanced and $\hh(c) = 2$, we have $N(c) \subset V_1 \cup V_3$.
    But since $T$ is WTD, $|N(c) \cap V_1| = 1$ by Theorem~\ref{thm. char. of WTD delt. trees}.
    Thus $\deg(c) = 1$, so $c$ is a leaf, contradicting that $\hh(c) = 2$.

    Next, assume by way of contradiction that for all $u \in V_2$, we have $|N_T(u) \cap V_3| \geq 2$.
    Let $T' = (V',E')$ be the subgraph of $T$ induced by the set $V_2 \cap V_3$.
    Since $T$ is a balanced tree, there is no edge between vertices of same height.
    Hence every edge in $T'$ is incident to a vertex in $V_2(T)$ and a vertex in $V_3(T)$; hence $T'$ is a bipartite graph with partite sets $V_2(T)$ and $V_3(T)$.
    Thus we get 
    $$
    |E'| = \sum_{v \in V_2(T)} \deg_{T'}(v) = \sum_{v \in V_2(T)} |N_T(v) \cap V_3(T)|.
    $$
    Then by assumption, we get
    \begin{align*}
        |E'| &= \sum_{v \in V_2(T)} |N_T(v) \cap V_3(T)|\\
        &\geq \sum_{v \in V_2(T)} 2\\
        &= 2 \cdot |V_2(T)|\\
        &> |V_2(T)| + |V_3(T)|. \tag{Lemma~\ref{lem. heighestSmallest}}
    \end{align*}
    By Lemma~\ref{lem. subgraph V32 is a tree}, $T'$ must be a tree.
    Hence we have 
    $$
    |E'| = |V'| - 1 = |V_2(T)| + |V_3(T)| - 1
    $$ 
    contradicting the inequality $|E'| > |V_2(T)| + |V_3(T)|$.
\end{proof}

Now we are ready to give our constructive characterization of WTD balanced trees of height 3.

\begin{theorem}\label{thm. constructing hh = 3 delt. trees}
    Let $T$ be an WTD balanced tree of height 3.
    Then either $T \cong P_6$, or there exists a sequence of vertices $v_1,\dots,v_k$ in $V$ such that $T \cong \mathcal{O}(P_6,(v_1,\dots,v_k))$.
\end{theorem}

\begin{proof}
    We prove the claim by minimal counterexample.
    So assume that $T$ is an WTD balanced tree of height 3 that can not be obtained by applying $\mathcal{O}$ to any other WTD balanced tree of height 3.
    By Lemma~\ref{lem. LeafAdding}, we may assume that every support vertex of $T$ has exactly one leaf attach to it.
    By Corollary~\ref{cor. hh = 2 of upper deg 1 exists}, there exists some vertex $u \in V_2$ such that $|N(u) \cap V_3| = 1$.
    Let $r \in N(u) \cap V_3$.
    
    Suppose that $\deg(r) > 2$, and set $T'$ to be the subgraph of $T$ induced by $V\setminus \br_r(u)$. 
    By assumption, the subgraph of $T$ induced by the set $\br_r(u)$ is isomorphic to $P_2$, so we get $T \cong \mathcal{O}(T',r)$.
    But, $T'$ is also an WTD balanced tree of height 3 by Fact~\ref{fact. O123 preserve WTD-ness}, a contradiction.

    Now suppose that $\deg(r) = 2$.
    If $|V_3| > 1$, then set $N(r) =: \set{u,u'}$.
    Now let $T'$ be a subgraph of $T$ induced by $V\setminus \br_{u'}(r)$ (this time, the subgraph induced by $\br_{u'}(r)$ is isomorphic to $P_3$), we get $T \cong \mathcal{O}(T',u')$, a contradiction since $T'$ is an WTD balanced tree of height 3 by Fact~\ref{fact. O123 preserve WTD-ness}.
    If $|V_3| = 1$, then $T \cong P_6$ by the assumption made by using Lemma~\ref{lem. LeafAdding}.
\end{proof}

\begin{example}
    Consider the following tree in Figure~\ref{fig: const delt. tree}.
    \begin{figure}[ht]
        \centering
        \includegraphics{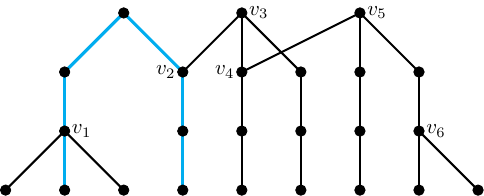}
        \caption{An WTD balanced tree of height 3, $T$.}
        \label{fig: const delt. tree}
    \end{figure}
    Noticing that the cyan edges form a $P_6$, one can see that $T \cong \mathcal{O}(P_6,(v_1,v_1,v_2,v_3,v_4,v_5,v_6))$.
    %Notice that other than the leaves in the sequence $(v_1,v_1,v_2,v_3,v_4,v_5,v_6)$, the inverse ordering of the vertices in the sequence gives you an order of finding height 2 vertices $u$ with $|N(u) \cap V_3| = 1$ applying Corollary~\ref{cor. hh = 2 of upper deg 1 exists}.
\end{example}

\printbibliography

\end{document}